\documentclass[reqno,oneside,a4paper,11pt]{amsart}
\usepackage[utf8]{inputenc}
\usepackage{mathtools} %superimpose objects
\usepackage{amsaddr}
\usepackage{amsmath,amsthm,amsbsy,amssymb, comment, booktabs, float}
\usepackage{arydshln} %DASHLINE IN TABLES
\usepackage{transparent} 
\usepackage{physics} 

\usepackage{xparse}     % <----- for '\NewDocumentCommand' command (necessary)

\NewDocumentCommand{\mybar}{ O{0.860} O{0.3pt} m }{%  <---- Set the default values here
    \mathrlap{\hspace{#2}\overline{\scalebox{#1}[1]{\phantom{\ensuremath{#3}}}}}\ensuremath{#3}
}

\newcommand{\cS}{\mathcal{S}}
\newcommand{\cR}{\mathcal{R}}

\newcommand{\cH}{\mathcal{H}}
\newcommand{\cM}{\mathcal{M}}

\newcommand{\cT}{\mathcal{T}}

\newcommand{\dens}{d} % the density \dens_p(u) for u mod p

\makeatletter
\renewcommand{\pod}[1]{\allowbreak\mathchoice
  {\if@display \mkern 18mu\else \mkern 2mu\fi (#1)}
  {\if@display \mkern 6mu\else \mkern 2mu\fi (#1)}
  {\mkern4mu(#1)}
  {\mkern4mu(#1)}
}
\makeatother

\usepackage{bm}

\newcommand{\legendre}[2][p]{\ensuremath{\left( \frac{#2}{#1} \right) }}
\newcommand{\sdfrac}[2]{\mbox{\small$\displaystyle\frac{#1}{#2}$}}

\newcommand{\ZZ}{\mathbb{Z}}

\newcommand{\FF}{\mathbb{F}}

\theoremstyle{plain}
\newtheorem{theorem}{Theorem}

\newtheorem{lemma}{Lemma}[section]
\newtheorem{proposition}{Proposition}[section]

\theoremstyle{remark}
\newtheorem{remark}{Remark}[section]

\theoremstyle{definition}

\usepackage[font={up,small}]{caption}
\usepackage{xcolor} 
\definecolor{orange}{rgb}{1,0.5,0}
\definecolor{Ggreen}{rgb}{0.,0.575,0.0128}
\definecolor{Bblue}{rgb}{0.016,.132,0.81}
\usepackage[hyphens]{url}

 \usepackage[backref=page]{hyperref}
\hypersetup{
    colorlinks = true,
linkcolor={red},
urlcolor={Bblue},
citecolor={Ggreen},
urlcolor = {Bblue},
citebordercolor = {0.33 .58 0.33},
 linkbordercolor = {0.99 .28 0.23},
 breaklinks=true
}
\usepackage{cite}% condensed citations

\usepackage{subfigure}  % package for subfigure formatting
\usepackage[pdftex]{graphicx}
\graphicspath{{IMAGES/}{IMAGESFR/}{IMAGESH/}{./}}

\makeatletter
\def\mysequence#1{\expandafter\@mysequence\csname c@#1\endcsname}
\def\@mysequence#1{%
  \ifcase#1\or left\or right\or altceva\else\@ctrerr\fi}
\makeatother

\usepackage{tikz}
\usetikzlibrary{arrows.meta,    % new
                bending,        % new
                tikzmark}

\makeatletter
\@namedef{subjclassname@2020}{\textup{2020} Mathematics Subject Classification}
\makeatother
\begin{document}

% Hexagonal Tilings of the Plane with Integers.\\
% Hexagonal Tilings with Integers.\\
% Covering the Hexagonal Tilings with Integers.\\
% Covering Hexagonal Lattice with Integers.\\
% Hexagonal Lattice Covering with Integers.\\
% Building a hexagonal sandpile
\title[A lozenge triangulation of the plane with integers]
{A lozenge triangulation of the plane with integers}

\author{Raghavendra N. Bhat}
\address[Raghavendra N. Bhat]{Department of Mathematics, University of Illinois at Urbana-Champaign, 1409 West Green 
Street, Urbana, IL 61801, USA}
\email{rnbhat2@illinois.edu}

\author[Cristian Cobeli]{Cristian Cobeli\textsuperscript{*}}
\address[Cristian Cobeli]{``Simion Stoilow'' Institute of Mathematics of the Romanian Academy,~21 Calea Grivitei Street, P. O. Box 1-764, Bucharest 014700, Romania}
\email{cristian.cobeli@imar.ro}

\thanks{\textsuperscript{*}Corresponding author: Cristian Cobeli: \texttt{cristian.cobeli@gmail.com}}

\author{Alexandru Zaharescu}
\address[Alexandru Zaharescu]{Department of Mathematics, University of Illinois at Urbana-Champaign, 1409 West Green 
Street, Urbana, IL 61801, USA,
% .}
%
% \address[Alexandru Zaharescu]{
and 
``Simion Stoilow'' Institute of Mathematics of the Romanian Academy,~21 
Calea Grivitei 
Street, P. O. Box 1-764, Bucharest 014700, Romania}
\email{zaharesc@illinois.edu}

% 52C20 Tilings in 2 dimensions (aspects of discrete geometry)
% 11B37 Recurrences {For applications to special functions, see 33-XX}
% 11B50 Sequences (mod m)
% 11B39 Fibonacci and Lucas numbers and polynomials and generalizations
% 11Y16 Number-theoretic algorithms; complexity
% \subjclass[2020]{Primary 11B37; Secondary 11B39, 11B50.}
% 11K36 Well-distributed sequences and other variations
% 11Bxx Sequences and sets
% 11B99 None of the above, but in this section

% 11B99 11B85
% 11Bxx Sequences and sets
% 11B99 None of the above, but in this section
% 05A05 Permutations, words, matrices
% 68R15 Combinatorics on words
% 11M41 Other Dirichlet series and zeta functions 


\begin{abstract}
We introduce and study a three-folded linear operator depending on three parameters
that has associated a triangular number tilling of the plane.
As a result the set of all triples of integers is decomposed in classes of equivalence organized in
four towers of two-dimensional triangulations. 
We provide the full characterization of the represented integers belonging to each network
as families of certain quadratic forms. 
We note that one of the towers is generated by a germ that produces a covering of the plane 
with {Löschian} numbers.
\end{abstract}

%%% ----------------------------------------------------------------------
\maketitle
%%% -------------------------------------------------------

% \section{Introduction and Summary of previous results}
\section{Introduction}\label{Introduction}
The study of discrete dynamical systems has gained attention because of
their capacity to model complex phenomena through the iterative 
application of simple rules. 
Understanding the common features and applications of these systems can provide valuable insights into the broader field and their practical implications.

A few examples that have been discussed again in recent times are
related to various aspects of the Ducci game rule (see~\cite{CCZ2000,CZ2014,CPZ2016}),
the study of phenomena occurring in Pascal-like triangles (see~\cite{Pru2022,CZ2013}) and to the contrasting patterns
produced by the iteration of the PG~\cite{BCZ2023,CZZ2013} operator that calculates the gaps between
the neighbor elements of a sequence in relation to the Proth-Gilbreath 
Conjecture~\cite{Pro1878,Guy2004,Guy1988,Gil2011}.
% Mon1994,

 In this context, we introduce here the following three-folded linear operator that has nonlinear characteristics.
Let $H=H(x,y,x)$ be the set $H=\{H',H'',H'''\}$ of transformations
that each leaves two components unchanged while to the negative third adds $1$ and the other two, that is,
\begin{equation}\label{eqHHH}
   \begin{split}
    H'(x,y,z)&=(-x+1+y+z,y,z),\\
        H''(x,y,z)&=(x,-y+1+z+x,z),\\
        H'''(x,y,z)&=(x,y,-z+1+x+y).
   \end{split}
\end{equation}

As we will show, the repeated application of any combination of these operators
produce a set of interconnected integers with remarkable properties. 
A special instance of these are the {Löschian} numbers.
Introduced in an economic model~\cite[Chapter 10]{Loc1940},
in which the producers and consumers are organized in a
convenient hexagonal network, {Löschian} numbers appear also in the more 
general geographical Central Place theory~\cite{Bat2013, Mar1977, BGNR2023}.
Their special properties make them useful in other 
theoretical~\cite{CD2014,Gol1935,Gol1937,CRS1999,KRNG2024} or 
practical contexts 
(see~\cite[Chapter 3]{Pet1998}, \cite{AA1989,Pet1998,Rus2017}
and the references mentioned there).

We denote by $H^{[n]}$, $n\ge 0$, the iterations of~$H$, which are the successive 
compositions of~$H$, where at any step any of $H',H'',H'''$ is applied. 
Thus, if $(a,b,c)\in\ZZ^3$, then
$H^{[n]}(a,b,c)$ contains $3^n$ triples of integers, not necessarily distinct.
For example, $H^{[0]}(1,2,3) = \{(1,2,3)\}$ and 
$(5,9,5)$ is one of the $243$ elements of $H^{[5]}(1,2,3)$, because
\begin{equation}\label{eqExample}
  \begin{split}
    (1,2,3) & \xrightarrow{\text{ $H'$ }} (5,2,3)
             \xrightarrow{\text{ $H'''$ }} (5,2,5)
             \xrightarrow{\text{ $H''$ }} (5,9,5)\\
             &
             \xrightarrow{\text{ $H'$ }} (10,9,5)
             \xrightarrow{\text{ $H'$ }} (5,9,5).      
  \end{split}
\end{equation}
Let $\cT_H(a,b,c)$ denote the union of all these triples, that is,
\begin{equation}\label{eqT}
    \cT_H(a,b,c) := \bigcup_{n\ge 0} H^{[n]}(a,b,c)\,.
\end{equation}
We note that if the union were disjoint, then it would be possible for a triple to repeat, because for example if the sum of two components, let's say $a$ and $b$, is odd and the other is $c=\frac{a+b+1}{2}$, then 
$H'''(a,b,c)=(a,b,c)$. Other type of repetitions also occur, all due
to the symmetries created by all possible orders in which operators~\eqref{eqHHH} are applied.

We say that an integer is \textit{represented} by $H^{[n]}(a,b,c)$ if it appears as 
a component of a triple in $\cT_H(a,b,c)$. 
Since a represented integer can appear multiple times in the triples from $\cT_H(a,b,c)$,
we call the \textit{length} of $m$ as
the smallest possible number of compositions required to go from $(a,b,c)$ to 
a triple that has $m$ as a component.
Thus, following the sequence in~\eqref{eqExample}, we see that $1,2$ and $3$
have length $0$,
then $5$ has length $1$, and it can be checked that~$9$ has length $3$, as there is no shorter
path to obtain it. However, the length of $10$ is not $4$, but also~$3$, since $10$ is represented
in this shorter branch:
$ (5,2,3) \xrightarrow{\text{ $H''$ }} (5,7,3)
             \xrightarrow{\text{ $H'''$ }} (5,7,10)$.

Let $\cR_H(a,b,c)$ denote the set of all integers represented by $H(a,b,c)$, that is,
\begin{equation*}
    \cR_H(a,b,c) := \big\{ m \in \{x,y,z\} :  (x,y,z) \in\cT_H(a,b,c)\big\}\,.
\end{equation*}
Given a triple $(a,b,c)\in\ZZ^3$, from the way definitions~\eqref{eqHHH} are introduced, one should not expect that all or at least most integers are represented in
$\cR_H(a,b,c)$.
And the reason is not the fact that two variables are added and only one subtracted in formulas~\eqref{eqHHH},
and therefore depending on the signs of $a, b, c$ very small or very large numbers would not be represented. 
The real main factor is actually the addition of $1$ in formulas~\eqref{eqHHH}, 
which always makes that only very few small numbers are represented.
The following theorem proves this fact for numbers less than any given threshold.
%%%%%%%%%%%%%%%%%%%%%%%%%%%%%%%%%%%%%%%%%%%%%%%%%%%%%%%%%%%%
 \begin{theorem}\label{TheoremA}
Let $M$ be a fixed integer. Than, 
$\cR_H(a,b,c)\cap (-\infty,M]$ is finite for any $a,b,c\in\ZZ$. 
% \begin{equation}\label{eqFinite}
%     \cR_H(a,b,c)\cap (-\infty,M] \ \ \text{is finite.}
% \end{equation}
\end{theorem}

Let us note that the sets of represented numbers by $H$ are actually closely related to one another 
by the equality
\begin{equation}\label{eqTranslated}
    \cR_H(a+h,b+h,c+h) = \cR_H(a,b,c) +h
\end{equation}
for any $a,b,c,h\in\ZZ$.
Indeed, this follows since if the triple $(A,B,C)$ is obtained from $(a,b,c)$ through the 
sequence of operations $H^{[n]}$, then through exactly the same sequence of operations we obtain (see Lemma~\ref{LemmaTR})
\begin{equation*}%\label{eqTranslated}
    (a+h,b+h,c+h) \xrightarrow{\text{ $H^{[n]}$ }} (A+h,B+h,C+h)\,.
\end{equation*}
In Section~\ref{SectionGerms} we show that essentially there are only two distinct
sets of representatives, namely
$\cR_H(0,0,0)$ and $\cR_H(0,1,1)$,  
% that are generated by the two \textit{germs} $(0,0,0)$ and $(0,1,1)$, 
and all the others are obtained by translations~\eqref{eqTranslated}. 
A combined graphical representation of the \textit{fundamental sets} 
$\cR_H(0,0,0)$ and $\cR_H(0,1,1)$
and the set of triples
$\cT_H(0,0,0)$ and $\cT_H(0,1,1)$
is shown in Figure~\ref{FigureTwoGerms}.

These tessellations of the plain with integers have
 the particular characteristic that in any basic adjacent triangles that together form a \textit{lozenge} (diamond consisting of any four close circles positioned in such a way that each circle is adjacent to at least two of the remaining three) 
 the sum of the numbers in the nodes on the longer diagonal 
 is with~$1$ larger than the sum of the numbers on the shorter diagonal.
% HexSiz%%%%%%%%%%%%%%%%%%%%%%%%%%%%%%%%%%%%%%%%%%%%%%%%%%%%%%%%%%%%%%%%%%%%
% https://sage.syzygy.ca/jupyter/user/sucodru/notebooks/SQUARE_PRIMES/Hexagon%20Filling%20Fibonacci.ipynb
% some manual editing, i.e., scale, rotation, inclusion of the colorbar
\begin{figure}[ht]
%  \centering
%  \mbox{
%  \subfigure{
    \includegraphics[width=0.49\textwidth]{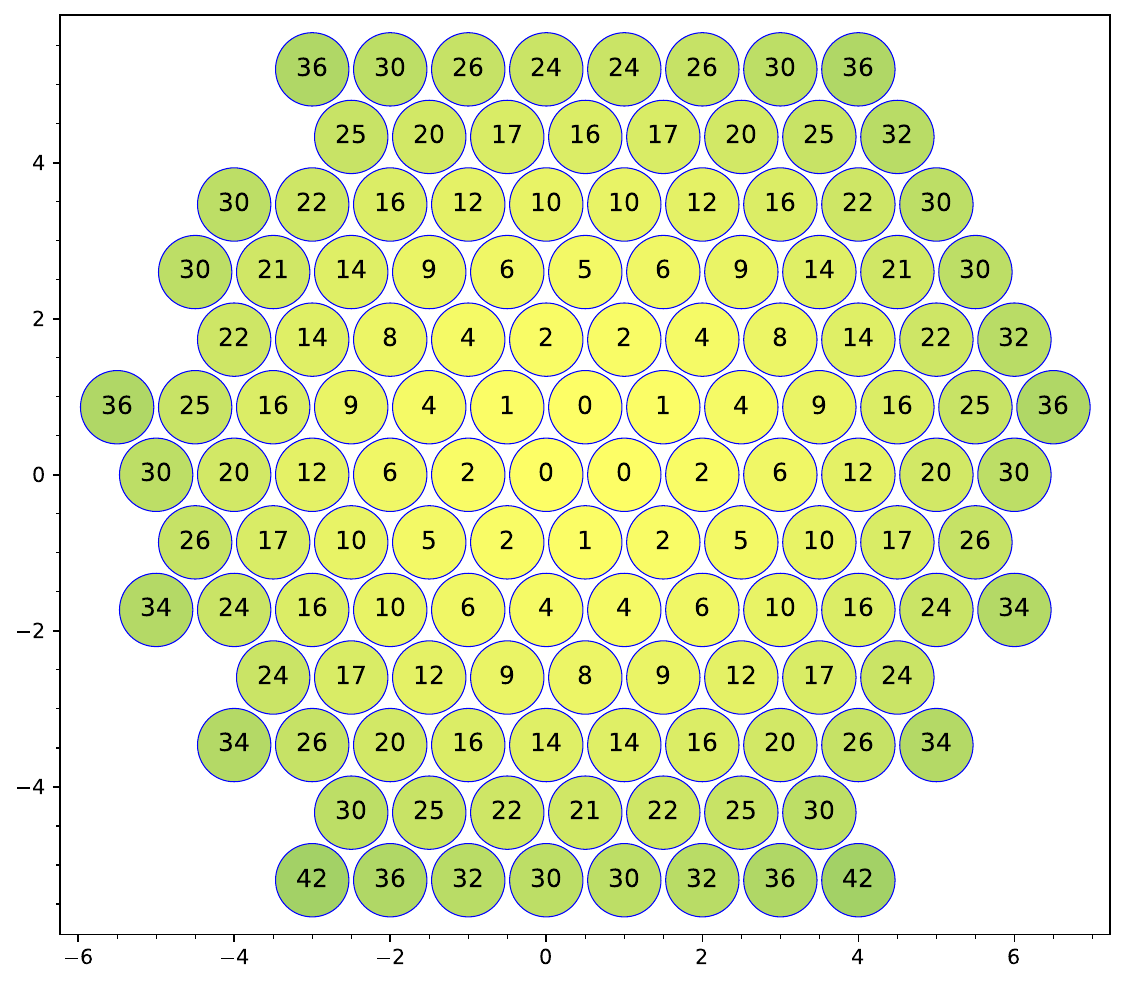}
%  \label{Fig2Examples1}
%  }%\quad
%  \subfigure{
    \includegraphics[width=0.49\textwidth]{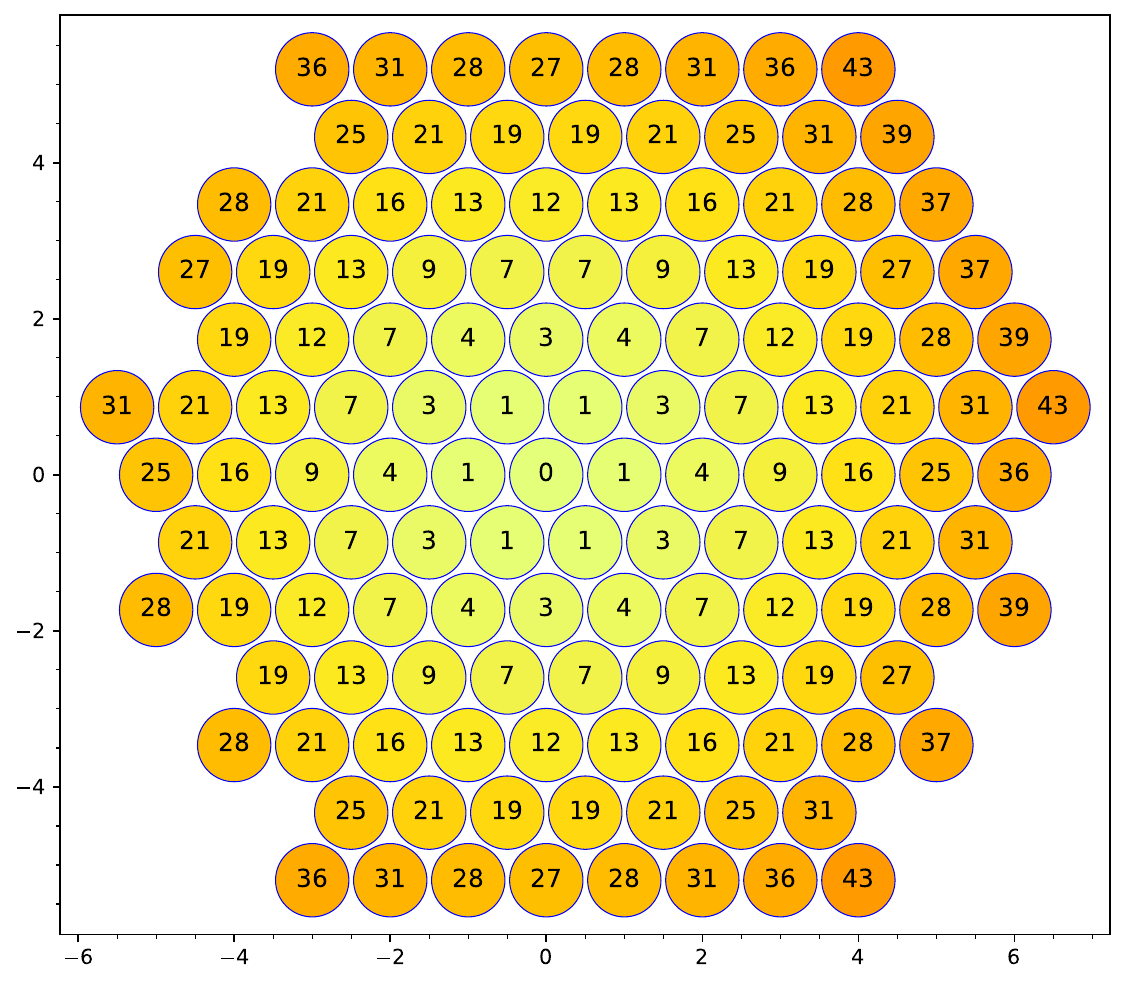}
%  \label{Fig2Examples2}
%  }
%  }
\caption{The geometrical representation of the core of $\cR_H(0,0,0)$ and $\cR_H(0,1,1)$
as the sets of nodes in the tiling of the plane with triples in
$\cT_H(0,0,0)$ and $\cT_H(0,1,1)$.
}
 \label{FigureTwoGerms}
 \end{figure}
 
Lozenge tilings of the plane or of some distinguished
domains with special characteristics have been intensively studied~\cite{Ciucu2009,CL2019,CLR2021,CECZ2001,Buy2022}, in part due to the connection with practical applications of a random tiling
model for two dimensional electrostatics~\cite{Ciucu2005, CF2023}.
While the main focus in these mentioned works is the evaluation of the number of distinct tilings,
our objectives include characterizing the weights in the nodes of a network generated by a certain 
triple, classifying the networks of numbers and highlighting certain unique paths that connect the network nodes.

Once we observe that any integer $a$ is represented in $\cR_H(*,*,*)$, 
with $a$ being one of the stars, the natural problem is to find
an efficient algorithm to decide whether or not a positive $a$ belongs to 
$\cR_H(0,0,0)$, $\cR_H(0,1,1)$, $\cR_H(1,0,1)$ or $\cR_H(1,1,0)$. 
For example, one finds that  $2024\not\in \cR_H(0,1,1)$,
but $2023\in \cR_H(0,1,1)$
(see the graphical representation in Figure~\ref{Fig2Examples}). For~$2023$, there are infinitely many 
nontrivial paths to get from $(0,1,1)$ to a triple that represents
$2023$, but one of the shortest takes $99$ steps.
Such a path passes by the intermediate triple $(1089,1156,1123)$ through 
the sequence of operations:
\begin{equation*}%\label{eqTranslated}
    (0,1,1)  
    \xrightarrow{
    \text{ $\big(H''\circ H'''\circ H'\big)^{[22]}$ }} 
    (1089,1151,1123),
\end{equation*}
and then, changing the direction, reaches the target through:
\begin{equation*}%\label{eqTranslated}
    (0,1,1)  
    \xrightarrow{
    \text{ $\big(H'''\circ H''\circ H'\big)^{[11]}$ }
    \circ
    \text{ $\big(H''\circ H'''\circ H'\big)^{[22]}$ }} 
    % (1089,1151,1123)
    (1956,1939,2023).
\end{equation*}

% \begin{remark}
% The symmetries of germs $(0,0,0)$ and $(0,1,1)$ imply the following invariance 
% to permutations of the fundamental sets of triples of representatives.

%     1. If $(a,b,c)\in\cT_H(0,0,0)$, then any permutation of $(a,b,c)$ belongs to
%     $\cT_H(0,0,0)$.

%     2. If $(a,b,c)\in\cT_H(0,1,1)$, then $(a,c,b)$ belongs to
%     $\cT_H(0,1,1)$.
% \end{remark}

Searching for common patterns among the sets of triples $\cT_H(a,b,c)$, one
can sometimes notice that starting from two different triples 
$(a_1,b_1,c_1)$ and $(a_1,b_1,c_1)$,
and performing an identical sequence of operations, % from $\{H',H'',H'''\}$, 
it can happen that one of the components of the resulting triples 
ends up being the same.
For example, $G=H'''\circ H''\circ H'''\circ H'$ transforms
$(1, 3, 6)$ and $(6, 7, 9)$ into two triples that have the third component
equal to $17$:
% (1, 3, 6)
% (9, 3, 6) H1
% (9, 3, 7) H3
% (9, 14, 7) H2
% (9, 14, 17) H3
% ......................
% (9, 13, 17) H2
% (9, 13, 6) H3
% (11, 13, 6) H1
% (11, 13, 19) H3
\begin{equation}\label{eqExample2}
  \begin{split}
    (1, 3, 6) &\xrightarrow{\text{ $H'$ }} (9, 3, 6)
             \xrightarrow{\text{ $H'''$ }} (9, 3, 7)
             \xrightarrow{\text{ $H''$ }} (9, 14, 7)
             \xrightarrow{\text{ $H'''$ }} (9, 14, 17)\\
    (6, 7, 9) &\xrightarrow{\text{ $H'$ }}(11, 7, 9)
             \xrightarrow{\text{ $H'''$ }} (11, 7, 10)
             \xrightarrow{\text{ $H''$ }} (11, 15, 10)\\
             &
             \xrightarrow{\text{ $H'''$ }} (11, 15, 17)\,.
  \end{split}
\end{equation}
% (6, 7, 9)
% (11, 7, 9) H1
% (11, 7, 10) H3
% (11, 15, 10) H2
% (11, 15, 17) H3
% ......................
% (11, 14, 17) H2
% (11, 14, 9) H3
% (13, 14, 9) H1
% (13, 14, 19) H3
Continuing from this point, the search reveals that even more triples 
obtained in a similar way satisfy the same property. 
Thus, in the example above, one  even finds an infinite sequence of such
triples obtained through a periodic series of transformations.
Indeed, if $F=H'''\circ H'\circ H'''\circ H''$, then the third component of 
the sequences
$\{F^{(k)}(9, 14, 17)\}_{k\ge 0}$ and $\{F^{(k)}(11, 15, 17)\}_{k\ge 0}$
are equal, and the sequence of these equal components starts with 
$17, 19, 27, 41,$
$ 61,$
$ 87, 119, 157, 201, 251, 307,\dots$,
the general formula for the general term being
$\{3k^2 - k+ 17\}_{k\ge 0}$.
In the following remark we note that this observation is universally valid.

%%%%%%%%%%%%%%%%%%%%%%%%%%%%%%%%%%%%%%%%%%%%%%%%%%%%%%%%%%%%
 \begin{remark}\label{Remark11}
Let $(a_1,b_1,c_1)$ and $(a_2,b_2,c_2)$ be triples of integers.
% Suppose there exists~$G$, a finite sequence of compositions of transformations
% $H', H'', H'''$ defined by~\eqref{eqHHH},
% such that 
Suppose $G$ is a composition of a finite sequence of operators from $\{H',H'',H'''\}$ and
$G(a_1,b_1,c_1) = (A,*,*)$ and $G(a_2,b_2,c_2) = (A,*,*)$, 
where $A\in\ZZ$, and the~$*$'s may represent any integer.
Then, there exists an infinite sequence of distinct integers $\{A_n\}_{n\ge 0}$ 
and a sequence $\{G_n\}_{n\ge 0}$, where each~$G_n$ is a sequence of compositions  of
$H', H'', H'''$, such that
\begin{equation*}
    G_n(a_1,b_1,c_1) = (A_n,*,*) = G_n(a_2,b_2,c_2)\ \ \text{ for $n\ge 0$.}
\end{equation*}
\end{remark}
We will revisit this remark in Section~\ref{subsection2D} after we will
obtain the characterization of the set of triples $\cT_H(a,b,c)$.

The next result shows that the represented numbers in both 
fundamental tessellations $\cT_H(0,0,0)$ and $\cT_H(0,1,1)$
cover all the residue classes modulo any prime $p\ge 5$.
The graphical representations $\mod p$ reveal intricate patterns 
(see Figures~\ref{FigGermsMod2331Ales} and~\ref{FigGermsMod2357}).
Note that while the density of the represented weights in residue classes is uniform, there are two particular residue classes 
where the density is either very small or very large compared to the others.

%%%%%%%%%%%%%%%%%%%%%%%%%%%%%%%%%%%%%%%%%%%%%%%%%%%%%%%%%%%%
 \begin{theorem}\label{TheoremC}
Let $p$ be prime and let $\dens_p(u)$ denote the limit density of 
the residue class $u = R \pmod p$
of the represented integers in $R\in \cR_H(a,b,c)$, for $u=0,1,\dots,p-1$.
Then:
% \noindent
% \vspace{1mm}

(1) If the germ is $(0,0,0)$, then the values of the density are:
\begin{itemize}\setlength\itemsep{5pt}
    \item [$\circ$]
If $p=2$, then $d_2(0) = 3/4$ and $d_2(1) = 1/4$.
    \item  [$\circ$]
If $p=3$, then $d_3(0) = 3/9$ and $d_3(1) = 3/9$ and $d_3(2) = 3/9$.
    \item  [$\circ$]
If $p\ge 5$ and $p\equiv 1\pmod 6$, then $d_p\big((p-1)/3\big) = (2p-1)/p^2$     and  $d_p(u) = (p-1)/p^2$ for $u=1,\dots,p-1$ and $\neq (p-1)/3$.
   % i.e., [7, 13, 19, 31, 37, 43, 61, 67, 73, 79, 97]
   
    \item  [$\circ$]
If $p\ge 5$ and $p\equiv 5\pmod 6$, then $d_p\big( (2p-1)/3\big) = 1/p^2$ and  $d_p(u) = (p+1)/p^2$ for $u=0,\dots,p-1$ and $\neq (2p-1)/3$.
    % i.e., [5, 11, 17, 23, 29, 41, 47, 53, 59, 71, 83, 89]
\end{itemize}

\vspace{1mm}
% \noindent
(2) If the germ is $(0,1,1)$, then the values of the density are:
\begin{itemize}\setlength\itemsep{5pt}
    \item  [$\circ$]
If $p=2$, then $d_2(0) = 1/4$ and $d_2(1) = 3/4$.
    \item  [$\circ$]
If $p=3$, then $d_3(0) = 3/9$ and $d_3(1) = 6/9$ and $d_3(2) = 0$.
    \item  [$\circ$]
If $p\ge 5$ and $p\equiv 1\pmod 6$, then $d_p(0) = (2p-1)/p^2$ and  $d_p(u) = (p-1)/p^2$ for $u=1,\dots,p-1$.
   % i.e., [7, 13, 19, 31, 37, 43, 61, 67, 73, 79, 97]

    \item  [$\circ$]
If $p\ge 5$ and $p\equiv 5\pmod 6$, then $d_p(0) = 1/p^2$ and  $d_p(u) = (p+1)/p^2$ for  $u=1,\dots,p-1$.
    % i.e., [5, 11, 17, 23, 29, 41, 47, 53, 59, 71, 83, 89]
\end{itemize}
\end{theorem}
%%%%%%%%%%%%%%%%%%%%%%%%%%%%%%%%%%%%%%%%%%%%%%%%%%%%%%%%%%%%%%%%%%%
% https://sage.syzygy.ca/jupyter/user/sucodru/notebooks/SQUARE_PRIMES/SP_SQUARE-PRIMES.ipynb
% some manual editing, i.e., scale, rotation, inclusion of the colorbar
% \setlength{\belowcaptionskip}{-10pt}
\begin{figure}[hb]
 \centering
 % {\transparent{0.94}
 \includegraphics[width=0.49\textwidth]{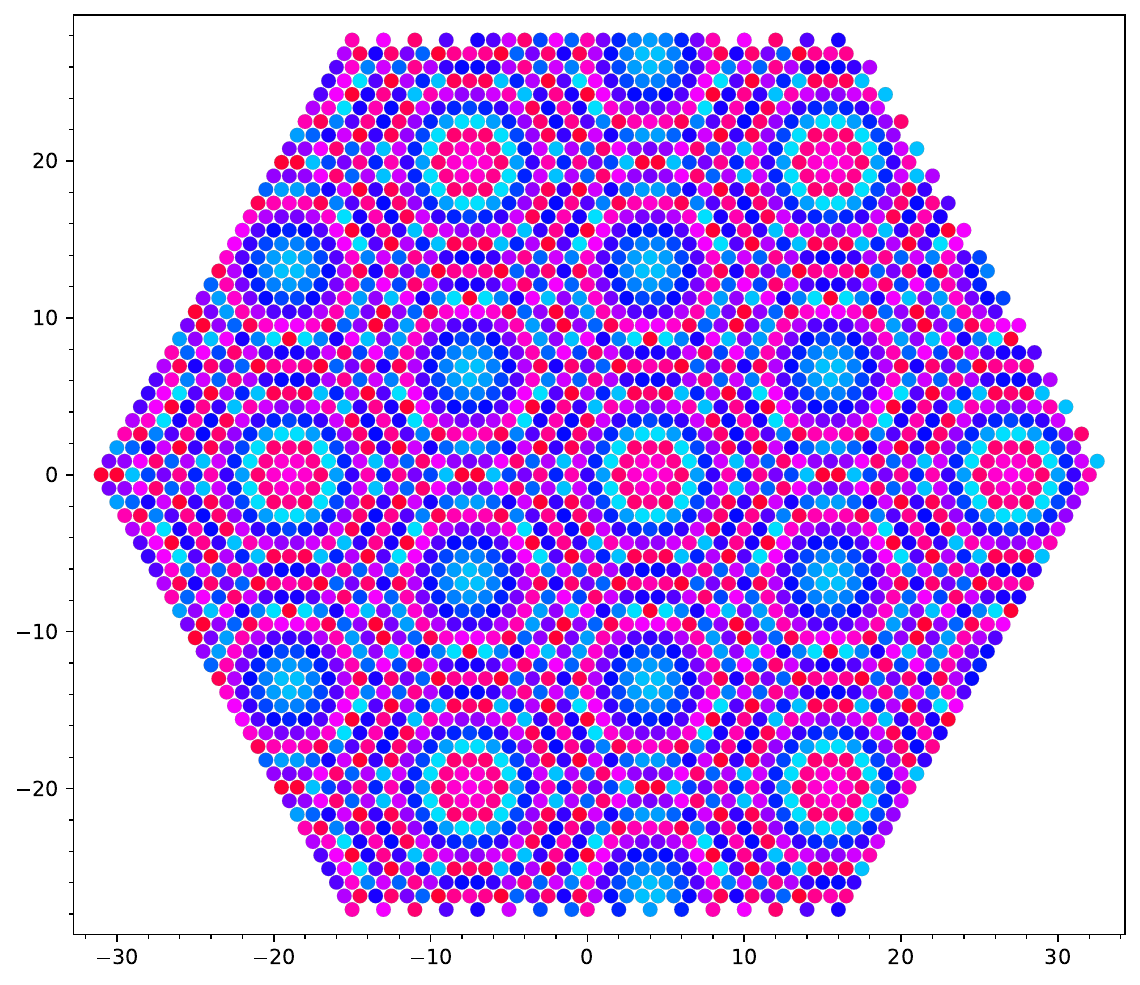}
  \includegraphics[width=0.49\textwidth]{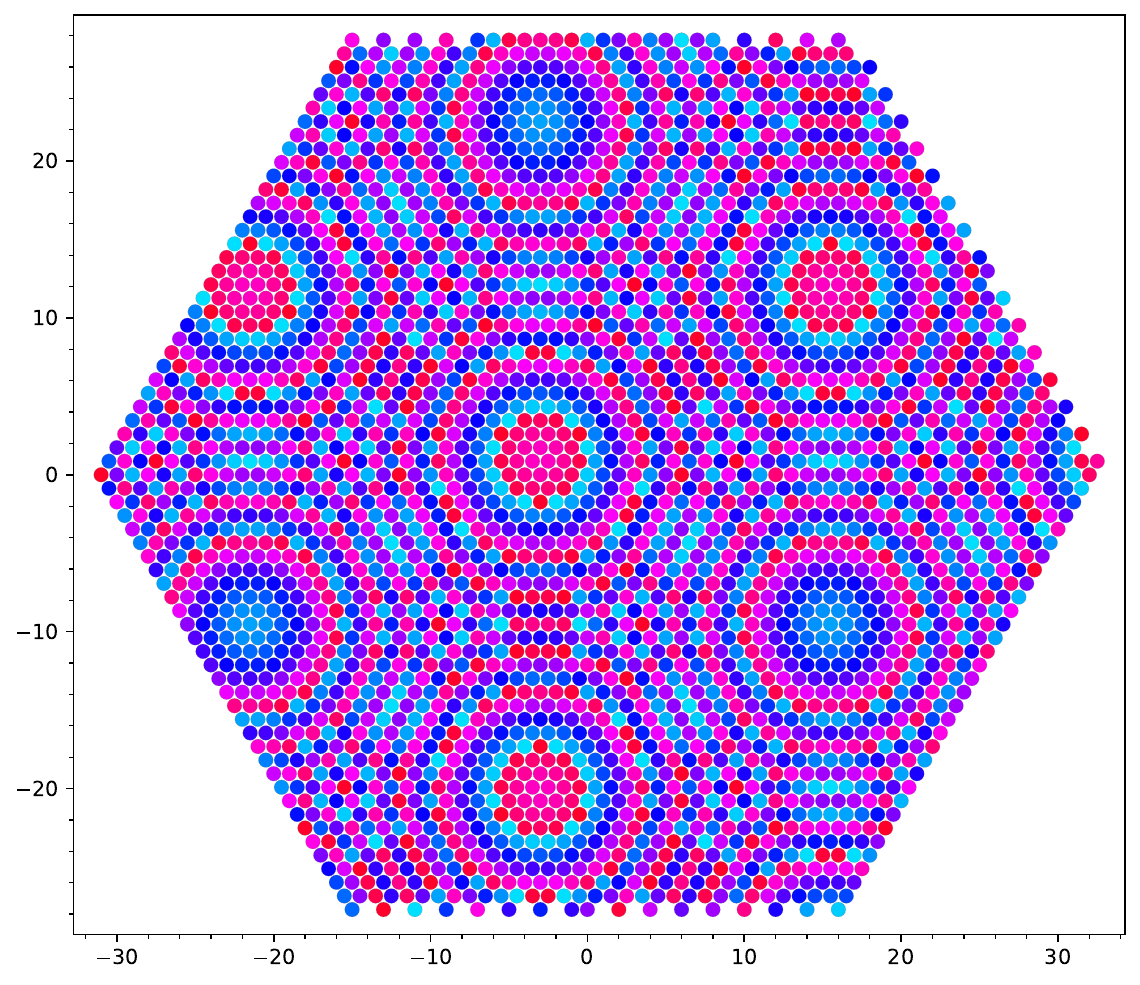}
\caption{A cut-off representation of the triangular networks generated by $(9,2,6)$ (left)
and $(1,8,3)$ (right). The weights are taken modulo $23$ on the left and modulo $37$ on the right.
In order to distinguish them, the residue classes are represented in distinct colors in each of the two cases.
}
 \label{FigGermsMod2331Ales}
 \end{figure}
Note that Theorem~\ref{TheoremC} allows us to distinguish or even precisely determine which tower a given 
triple $(a,b,c)$ belongs to.
For example, knowing additionally that $(a,b,c)$
appears in one of the basic tessellations $\cT_H(0,0,0)$,
$\cT_H(0,1,1)$, $\cT_H(1,0,1)$, $\cT_H(1,0,1)$, and one of
$a,b$ or $c$ is congruent to $2$ mod~$3$, then 
it would follow that $(a,b,c)\in\cT(0,0,0)$.
With less information, one needs to combine Theorem~\ref{TheoremC}
with finding a path to the center
(the triangle with the minimum weights) of the tessellation
and then travel by translation through the tower to find the basic germ (see the discussion in Section~\ref{subsectionEvaluationNegativeWeights}).

\medskip

The work is structured as follows.
In Section~\ref{SectionGerms} we present the first implications 
of employing the iterative composition of any combination of the operators
in $\{H',H'',H'''\}$, and then prove the existence of the 
four towers of tesselations corresponding to the fundamental germs.
Next, in Section~\ref{SectionLozenge}, we describe the geometric representation of the image of the composed operators $H$ 
as lozenge tilings of the plane 
with the represented integers placed as weights in the nodes.
Continuing on, in Section~\ref{SectionDensities}, we prove
Theorem~\ref{TheoremC}, which provides explicitely the densities of the
weights in residue classes modulo any prime number~$p$.
Our work concludes in Section~\ref{SectionRuleD} by achieving the parameterization of
all weights in a tessellation, which consequently supports the fact stated in 
Theorem~\ref{TheoremA} that only a finite number of weights can be found under any given threshold.
Additionally, we also present the procedure to be followed on the shortest route taken on the network nodes
from a specific triangle towards its origin.

%%%%%%%%%%%%%%%%%%%%%%%%%%%%%%%%%%%%%
\section{Properties of \texorpdfstring{$H$}{H}}\label{SectionGerms}

\begin{lemma}\label{LemmaTR}
    For any integer $a,b,c,h$ we have:
   \begin{equation}\label{eqTripleTranslations}
  \begin{split}
  \cT_H(a+h,b+h,c+h) &= \cT(a,b,c)+h\,,\\
  \cR_H(a+h,b+h,c+h) &= \cR(a,b,c)+h\,.
  \end{split}
\end{equation} 
\end{lemma}
\begin{proof}
We use the usual notations for the translations of a tuple and of a set,
by which $(a,b,c)+h:=(a+h,b+h,c+h)$, and $\cS+h$ represents the set obtained 
by adding $h$ to each element in the set $\cS$. 

The equality of the sets of represented integers in~\eqref{eqTripleTranslations}
follows from the equality of the sets of triples above, which in turn is implied 
by the definitions of the three analogous relations for $H',H''$ and $H'''$, of which, 
for exemplification, the first one is
       \begin{equation*}
  \begin{split}
  H'(a+h,b+h,c+h) &= \big(-(a+h)+1 + (b+h)+(c+h),b+h,c+h\big)\\
  &= (-a+1+b+c,b,c)+h \\
  &= H'(a,b,c)+h\,.
  \end{split}
\end{equation*} 
This completes the proof of the lemma.
\end{proof}

\begin{proposition}\label{PropositionInvolutionNonComute}
% The operators $H',H'',H'''$ are involutions but they
% do not commute one with each other.
We have:
\begin{enumerate}%[label=(\alph*)]
    \item[(1)] The operators $H',H'',H'''$ are involutions.
    \item[(2)] Any two distinct operators $H',H'',H'''$ do not commute.
    % \item $\big(H''\circ H'''\big)^{(3)}(x,y,z)=(x,z,y)$.
    \item[(3)] $H''\circ H'''\circ H'' = H'''\circ H''\circ H'''$.
    \item[(4)] $\big(H'\circ H''\big)^{[3]} = Id$.
\end{enumerate}
\end{proposition}
\begin{proof}
It is enough to check that $H'\circ H'=Id$ and 
$H'\circ H''\neq H'\circ H''$, as for the other operators
the analoguos relations follow by a rotation of the variables.
Also, let us note that the other permutations of the components 
as in part (3) or part (4) can be obtained in the same way through the corresponding composition 
of the analogous permuted operators.

Let $x,y,z$ be fixed. (1) For the fist part, we have:
\begin{equation*}%\label{eqInvolution}
  \begin{split}
  H'\big(H'(x,y,z)\big) &=  H'(-x+1+y+z,y,z) \\
  &= \big(-(-x+1+y+z)+1+y+z, y, z\big)\\
  &= (x,y,z).
  \end{split}
\end{equation*}

\noindent
(2) The necommutativity follows since
\begin{equation*}
  \begin{split}
  H'\big(H''(x,y,z)\big) &= H'(x,-y+1+z+x,z) \\
  &= \big(-x+1+(-y+1+z+x)+z, -y+1+z+x,z\big)\\
  &= (2-y+2z, 1-y+z+x,z),
  \end{split}
\end{equation*}
and
\begin{equation*}
  \begin{split}
  H''\big(H'(x,y,z)\big) &= H''(-x+1+y+z,y,z) \\
  &= \big(-x+1+y+z, -y+1+z+(-x+1+y+z),z\big)\\
  &= (1-x+y+z, 2-x+2z,  z),
  \end{split}
\end{equation*}
so that $\big(H'\circ H''\big)(x,y,z) = \big(H''\circ H'\big)(x,y,z)$ if and only if
$y=x$ and $z=x-1$.

\medskip\noindent
(4) As before, we have
\begin{equation*}\label{eqH23}
  \begin{split}
  H''\big(H'''(x,y,z)\big) &= (x, 2+2x-z,  1+x+y-z),
  \end{split}
\end{equation*}
which, applied twice, implies
\begin{equation*}\label{eqH23twice}
  \begin{split}
  \big(H''\circ H'''\big)^{[2]}(x,y,z) &= 
  \big(x, 2+2x-(1+x+y-z), \\
  &\phantom{(x,xX}1+x+(2+2x-z)-(1+x+y-z) \big)\\
  &= (x,1+x-y+z, 2+2x-y),
  \end{split}
\end{equation*}
and finally, for the third time,
\begin{equation*}\label{eqH23thrice}
  \begin{split}
  \big(H''\circ H'''\big)^{[3]}(x,y,z) &= 
  \big(x, 2+2x-(2+2x-y), \\
  &\phantom{(x,xX} 1+x+(1+x-y+z) -(2+2x-y)\big)\\
  &= (x,y, z).
  \end{split}
\end{equation*}
The same calculation with components and operators interchanged
proves formula (4), and also the fact that
$\big(H'\circ H''\big)^{[3]} = \big(H'\circ H'''\big)^{[3]} = Id$.

\medskip\noindent
(3)
By part (4) we know that $H''\circ H'''\circ H'' \circ H'''\circ H''\circ H''' = Id$, 
from which, composing on the right with $H''', H''$ and $H'''$, 
in that order, and using (1),
we obtain $H''\circ H'''\circ H'' = H'''\circ H''\circ H'''$.
These completes the proof of the proposition.
\end{proof}

\begin{theorem}\label{TheoremGerms}
Let $a,b,c\in\ZZ$. Then, there exists $h\in\ZZ$ such that either 
\begin{equation*}
   \begin{split}
    \cR_H(a,b,c) &= \cR_H(0,0,0) +h\quad\text{ or }\\
    \cR_H(a,b,c) &= \cR_H(0,1,1)+h = \cR_H(1,0,1)+h = \cR_H(1,1,0)+h.       
   \end{split}
\end{equation*}
\end{theorem}
\begin{proof}
    From Theorem~\ref{TheoremA} we know that the set $\cR_H(a,b,c)$ has a first element.
Then, let 
$m = \min \cR_H(a,b,c)$,
and let $(m,u,v)\in \cT_H(a,b,c)$ be a triple in which $m$ is represented. 
We do not know in advance on which position $m$ appears in the triple, 
but the other two situations in which $m$ would be in the second or 
the third position are treated the same.

Then there are two possibilities: $m\le v\le u$ or $m\le u\le v$.
Again, it is enough to discuss only one of the cases, the other being treated similarly by rotating the variables and the corresponding operators involved. 
Moreover, we will see that the cases blend together, because we will find that in fact $u=v$. 

Thus, without restricting the generality, we may also suppose that  $m\le u\le v$.
Since  $(m,u,v)\in \cT_H(a,b,c)$, it follows that 
the third component of $ H'''(m,u,v)$, which equals
$ -v +1+m+u$ belongs to $\cR_H(a,b,c)$. 
Then, since $m$ is the minimum, \mbox{$m\le -v +1+m+u$}, 
which implies $v\le u+1$. Therefore, either $v=u$ or $v=u+1$.

\medskip
\noindent
\texttt{Case $v=u+1$.}
Since $(m,u,u+1)\in\cT(a,b,c)$ it follows that 
$H'''(m,u,u+1)$ $=(m,u,u)$ is also in $\cT(a,b,c)$, so that it is 
enough to consider the case $u=v$.

\medskip
\noindent
\texttt{Case $v=u$.}
Starting from $(m,u,u)$, the following neighbor triples are also in $\cT_H(a,b,c)$:
\begin{equation*}%\label{eqP1}
    (m,u,u)  \xrightarrow{\text{ $H'''$ }} (m,u,m+1)
     \xrightarrow{\text{ $H''$ }}  (m,-u+2m+2,m+1).
\end{equation*}
The condition that $m$ is the minimum of $\cR_H(a,b,c)$ implies
$m\le -u+2m+2$, that is $u\le m+2$.
As a consequence, for $u$ are only three possible values, $m,m+1$ or $m+2$, which we will analyze next.

\begin{enumerate}
\setlength\itemsep{5pt}
    \item If $u=m$, then we have $(m,m,m)\in\cR_H(a,b,c)$.
    \item If $u=m+1$, since $v=u$, we ha that $(m,m+1,m+1)\in\cR_H(a,b,c)$.
    \item If $u=m+2$ then $(m,u,u)=(m,m+2,m+2)$ and
the following three-steps path also arrives at $(m,m,m)$:
\begin{equation*}%\label{eqP1}
    (m,m+2,m+2)  \xrightarrow{\text{ $H'''$ }} (m,m+2,m+1)
     \xrightarrow{\text{ $H''$ }}  (m,m,m+1)
     \xrightarrow{\text{ $H'''$ }} (m,m,m).
\end{equation*}    
% that is, $(m,m,m)\in\cR_H(a,b,c)$.
\end{enumerate}
\noindent
In conclusion, we have shown that, %if
provided
$m=\min \cR(a,b,c)$, then  either
\mbox{$(m,m,m)\in\cT_H(a,b,c)$}
or $(m,m+1,m+1)\in\cT_H(a,b,c)$.
By  Lemma~\ref{LemmaTR} it then follows that 
$(0,0,0)\in\cT_H(a-m,b-m,c-m)$ in the first case, and
$(0,1,1)\in\cT_H(a-m,b-m,c-m)$ in the second.
\end{proof}

The two fundamental sequences of the represented integers, 
the only ones that exist according to Theorem~\ref{TheoremGerms},
arranged in increasing order, are:
\begin{equation}\label{eqR12}
    \begin{split}
        \cR(0,0,0) & = \{ 0, 1, 2, 4, 6, 8, 9, 12, 14, 16, 20, 21, 22, 25, 30, 32, 36,% 40, 41, 42, 44, 49,
        \dots \},   \\
        \cR(0,1,1) & = \{ 0, 1, 3, 4, 7, 9, 12, 13, 16, 19, 21, 25, 27, 28, 31, 36, 37,% 39, 43, 48, 49
        \dots \}.
    \end{split}
\end{equation}
The numbers can be obtained starting with the germs $(0,0,0)$ and $(0,1,1)$
and then step by step using the definition, keeping or changing the operator from
$\{H',H'',H'''\}$, that is, moving around in all directions.
Another more efficient approach is utilizing the parametrization of the elements of the sequences 
presented in Section~\ref{SectionRuleD}.

The second of the two sequences in~\eqref{eqR12} are the 
\textit{{Löschian} numbers}~\cite[{\href{https://oeis.org/A003136}{A003136}}]{oeis}.
Named after August {Lösch}, the sequence 
is a bi-product of a study in the field of economics, 
regarding market development, population distribution, and the size of regions 
approximated using a honeycomb network~\cite[Chapter 10]{Loc1940}.
The ordered sequence $\cR(0,1,1)$ is abundant in properties 
(see~\cite{KRNG2024} and the references therein)
and besides algebra and number theory~\cite{CD2014,Gol1935,Gol1937,CRS1999}
it arises in very diverse contexts, such as 
a counter of the protein coats in a virus shell model~\cite[Chapter 3]{Pet1998},
at the confluences between art and mathematics~\cite{Rus2017},
a fractal generator in the theory of place geometry~\cite{AA1989, BGNR2023},
or in the geographical Central Place theory~\cite{Bat2013, Mar1977, BGNR2023}.

%%%%%%%%%%%%%%%%%%%%%%%%%%%%%%%%%%%%%%%%%%%%%%%%%%%%
\section{The lozenge representation of the operators \texorpdfstring{$H$}{H}}\label{SectionLozenge}
Given three integers $a, b, c$, we place them at the vertices of an equilateral triangle, called 
the \textit{base triangle}. Then, by fixing any two of the numbers $a, b, c$, we place in the fourth vertex of 
the \textit{lozenge}\footnote{A \textit{lozenge} is the union of two triangles, which
are identical to the base triangle and share a common side.} generated by the base
triangle (the endpoints of its smaller diagonal being the fixed numbers) the value 
of the corresponding operator from $\{H', H'', H'''\}$ 
(see Figure~\ref{FigRuleD} and the further in-depth description in Section~\ref{SectionRuleD}).

By applying this procedure to every pair of two numbers from $\{a, b, c\}$, we 
obtain three lozenges with a common base triangle, a star with three equilateral triangles 
built adjacent to the sides of the base triangle.
Continuing in the same way, we proceed with the newly obtained equilateral triangles,
and then repeat the process endlessly, resulting in a triangular tiling of the plane
with integers.

\begin{remark}\label{RemarkEssence}
In essence, the geometric interpretation of relations~\eqref{eqHHH}
is a $2$-dimensional tiling with integers in nodes,
which has the property that in any lozenge of it, 
the sum of the numbers on the long diagonal is always $1$ 
more than the sum of the numbers on the short diagonal.
\end{remark}
It is fundamental, and in Section~\ref{SectionRuleD} we will ascertain, that this construction is indeed consistent, 
meaning that any triangle, which is obtained in exactly three ways as the intersection of three lozenges is uniquely and unequivocally determined. 
One can check the outlined procedure starting with any triple of numbers from the adjacent circles in 
Figures~\ref{FigureTwoGerms} and~\ref{Fig2Examples}.

With this geometric representation it becomes clear that for any two triangles or triples of 
represented integers,
there are infinitely many paths and corresponding sequencers of operators taken from $\{H',H'',H'''\}$
that composed successively into a new operator $H$ connects one triple to the other.
In order to mark the association, on the set of all triples~$\ZZ^3$, 
let us define a relation according to which two triples are \textit{equivalent} 
if they belong to the same tiling.
This is the same as saying that the triples in any set $\cT_H(a,b,c)$ defined 
by~\eqref{eqT} are pairwise equivalent.
The equivalence classes of this relation are these parallel triangular tilings,
each of them being generated by any base triangle it contains.

It should be noted that formally, any triangle of adjacent numbers obtained by iterating the operators 
from $\{H',H'',H'''\}$ is associated with a specific order. 
The classification of triples in the equivalence classes $\cT_H$ has its unique features, 
which causes nearby triples in~$\ZZ^3$ to not always be equivalent. 
As a result, in general, the permutations of the triple $(a,b,c)$ do not necessarily belong to the 
same tiling $\cT_H(a,b,c)$.
%%%%%%%%%%%%%%%%%%%%%%%%%%%%%%%%%%%%%%%%%%%%%%%%%%%%%%%%%%%%%%%%%%%
% https://sage.syzygy.ca/jupyter/user/sucodru/notebooks/SQUARE_PRIMES/Hexagon%20Filling%20Fibonacci.ipynb#
% some manual editing, i.e., scale, rotation, inclusion of the colorbar
% \setlength{\belowcaptionskip}{-10pt}
\begin{figure}[ht]
 \centering
 % {\transparent{0.94}
 \includegraphics[width=0.43\textwidth]{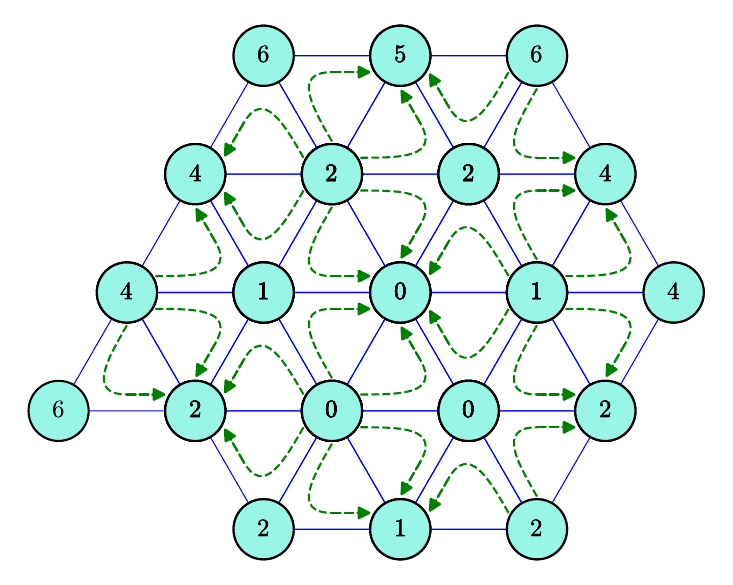}
 \qquad
  \includegraphics[width=0.43\textwidth]{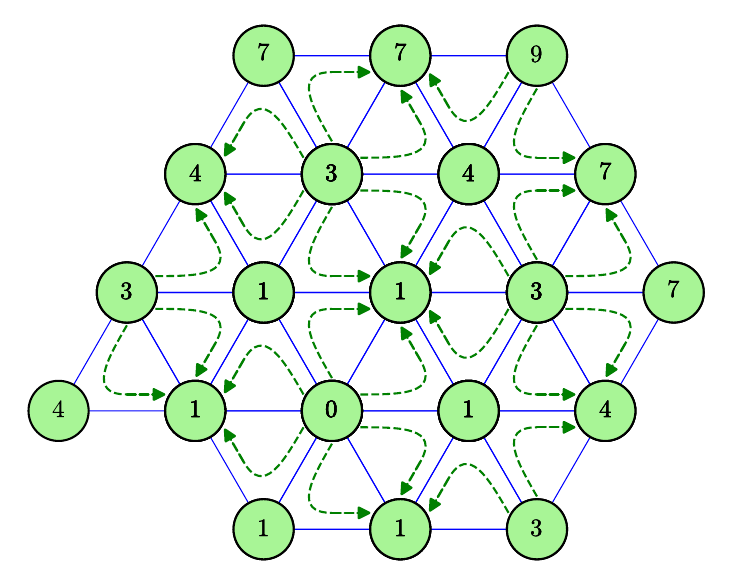}
\caption{The triples and how they appear oriented in the triangular networks generated by 
$(0,0,0)$ (left) and $(0,1,1)$ (right). Once can check that every triple 
that occurs appears only once in the left image and exactly six times in the right image.
}
 \label{FigureSpirals}
 \end{figure}

 \begin{remark}\label{RemarkCenters}
There is a difference between the rotation symmetries of the two triangular 
networks in Figure~\ref{FigureSpirals}.
In the network from the left-side, a spiral of lozenges is formed around a center consisting 
of the triangle with vertices $(0,0,0)$, while in the network on the right-side, 
although the generation also starts around a triangle, namely $(0,1,1)$, 
the network ends up having a single symmetry center, 
the node with weight $0$, instead of a triangle as in the left figure. 
\end{remark}

Additionally, we define an equivalence relation on the set of triples $\cT_H$ from~\eqref{eqT},
according to which $\cT_H(a_1,b_1,c_1)$ is equivalent to $\cT_H(a_2,b_2,c_2)$ if there exists an integer $h$
such that $\cT_H(a_1,b_1,c_1) = \cT_H(a_2,b_2,c_2) + h$.
In this way, $\ZZ^3$ is partitioned into equivalence classes formed by \textit{towers of tilings}, 
in which the elements are obtained from each other by translations, according to Lemma~\ref{LemmaTR}.
Furthermore, Theorem~\ref{TheoremGerms} states that every triple $(a,b,c)$ can be found in a unique tiling, 
and there are exactly four distinct towers. One of these, to which $\cT_H(0,0,0)$ belongs, is distinguished, 
while the other three have one of the three permutations of $(0,1,1)$ as the generating germ for their classes
of equivalence.

\begin{remark}\label{Remark16}
    The consequence of the difference between the two types of symmetry centers, 
as pointed out in Remark~\ref{RemarkCenters},
is the fundamental distinction between the two types of networks in the four towers.

Thus, any ordered triple $(a,b,c)$ appears exactly once in each triangular tiling $\cT_H(a,b,c)$
that is equivalent to $\cT_H(0,0,0)$
and exactly six times if it belongs to a tiling belonging to any of the other three towers generated by
$(0,1,1)$, $(1,0,1)$ or $(1,1,0)$.  
\end{remark}

%%%%%%%%%%%%%%%%%%%%%%%%%%%%%%%%%%%%%%%%%%%%%%%%%%%%
\section{The densities \texorpdfstring{$\pmod p$}{mod p} -- Proof of Theorem~\ref{TheoremC}}\label{SectionDensities}
In the following, we use the parametrization of the represented integers from 
Section~\ref{SectionRuleD}, Theorem~\ref{LemmaDL}.
%%%%%%%%%%%%%%%%%%%%%%%%%%%%%%%%%%%%%%%
\subsection{The germ \texorpdfstring{$(0,1,1)$}{(0,1,1)}}
Let $F(x,y)=x^2+xy+y^2$. We find the distribution of the 
residue classes of $F(x,y)$ modulo $p$ for $(x,y)$
in the modulo $p$ box  $[0,p-1]^2$.
For this, we need to calculate the number of solutions
$N_p(l)=\#\cM_p(l)$, where
\begin{equation}\label{eqNpl11}
	\cM_p(l) = \big\{(x,y) : x^2+xy+y^2 \equiv l \pmod p\big\},
\end{equation}
for any $l$, $0\le l \le p-1$. Then the limit density is $\dens_p(l) = N_p(l)/p^2$.
\smallskip

%%%%%%%%%%%%%%%%%%%%%%%%%%%%%%%%%%%%%%%%%%%%%%%%%%%%%%%%%%%%%%%%%%%%
% https://sage.syzygy.ca/jupyter/user/sucodru/notebooks/SQUARE_PRIMES/SP_SQUARE-PRIMES.ipynb
% some manual editing, i.e., scale, rotation, inclusion of the colorbar
% \setlength{\belowcaptionskip}{-10pt}
\begin{figure}[ht]
 \centering
 \includegraphics[angle=-90,width=0.99\textwidth]{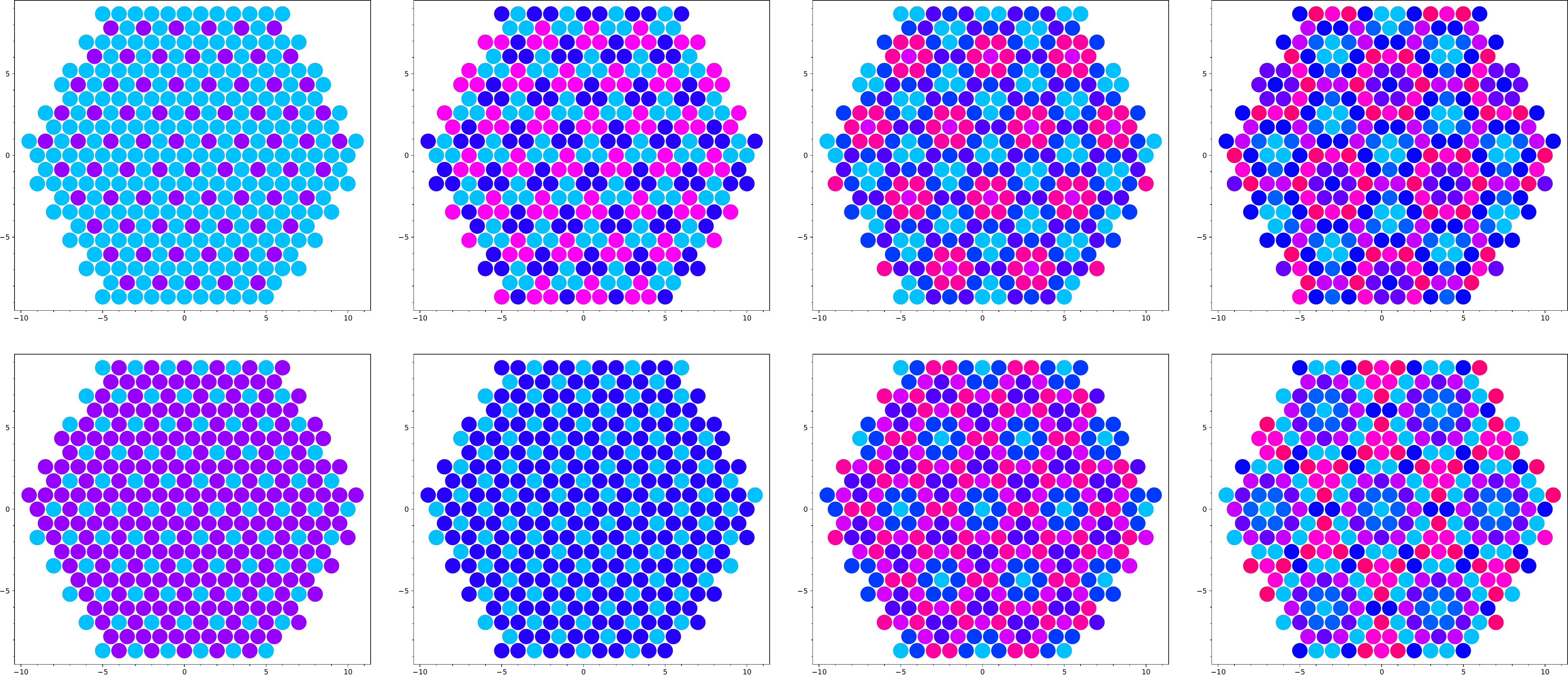}
\caption{Comparison between local symmetries of the two types of triangular networks.
On the top row is the network generated by $(0,0,0)$, and on the bottom row 
is the one generated by $(0,1,1)$.
In each of these, the weights are taken modulo $2, 3, 5$, and~$7$, and 
the residue classes are represented in distinct colors, 
so that the cases from Theorem~\ref{TheoremC} can be verified.
}
 \label{FigGermsMod2357}
 \end{figure}

\noindent
\texttt{Case $p=2$}.
Modulo $2$, the values of $F(x,y)$ for $x,y\in\{0,1\}$ are 
$ \begin{pmatrix} 0 & 1 \\ 1 & 1 \end{pmatrix} $,
% $0, 1, 1, 1$, 
so that 
$\dens_2(0) = 1/4$ and $\dens_2(1) = 3/4$.  
% 0  1  
% 1  1
% \smallskip

\noindent
\texttt{Case $p=3$}.
The values of $F(x,y) \pmod 3$, for $x,y=0,1,2$, are: 
$ \begin{pmatrix}
0 & 1 & 1 \\
1 & 0 & 1 \\
1 & 1 & 0 \\
\end{pmatrix} $.
% $0,1,1$ if $y=0$;
% $1,0,1$ if $y=1$; and $1,1,0$ if $y=2$.
Therefore, the value of the densities are:
$\dens_3(0)=3/3^2=1/3$, $\dens_3(1)=6/3^2=2/3$ and
$\dens_3(2)=0$.
% 0  1  1  
% 1  0  1  
% 1  1  0  

\noindent
\texttt{Case $p\ge 5$}.
\rule{0pt}{1.5em}% Empty vertical space
The handling of the congruence in~\eqref{eqNpl11} differs 
depending on whether $l$ is zero or not modulo $p$.
First, let us make a variable substitution to remove the term $xy$ from the left-hand side.
To do this, we multiply the congruence by $4$ and complete a square. Then, we replace the variables $x, y$ with $u, v$,
where $u = x+2^{-1}y$ and $v=2^{-1}y$.
With the simplified congruence, the problem reduces to
counting $N^*_p(l)=\#\cM_p^*(l)$, where
\begin{equation}\label{eqNstar011}
    \cM^*_p(l) = \big\{(u,v) : u^2+3v^2 \equiv l \pmod p\big\},
\end{equation}
since there is a one-to-one correspondence between the solutions of the two congruences, implying 
$N^*_p(l)=N_p(l)$.

\medskip
\noindent
\textbf{A.}
Suppose $l\equiv 0\pmod p$. Then the nature of the congruence 
in~\eqref{eqNstar011} depends on the fact that $-3$
is a quadratic residue or not. Firstly, if 
$\legendre{-3}=-1$\footnote{The Legendre symbol denoted by  $\legendre[p]{a}$ evaluates to $1$ if $a$ is a quadratic residue and to $-1$ if $a$ is a non-residue modulo $p$, for any $a$ relatively prime to $p$.}, the congruence has only a single solution, $(u,v)=(0,0)$.

Secondly, if $\legendre{-3}=1$, let $a$ be a square root
of $-3$. Then, for each $v\in\{1,\dots,p-1\}$, the congruence 
in~\eqref{eqNstar011} has exactly two distinct solutions
$(av,v)$ and $(p-av,v)$ because $2av\not\equiv 0\pmod p$.
Then, we have a total of $2p-1$ solutions in this case, including solution $(u,v)=(0,0)$.

\begin{remark}\label{RemarkEulerQRL}
We have $\legendre{-3}=1 $ if $p\equiv 1\pmod 6$ and
$\legendre{-3}=-1$  if $p\equiv 5\pmod 6$.
This result is obtained by combining Euler's 
criterion with the quadratic reciprocity law as follows:
\begin{equation*}
    \legendre{-3} = \legendre{-1}\legendre{3} 
    = (-1)^{\frac{p-1}{2}}\cdot (-1)^{\frac{p-1}{2}\cdot \frac{3-1}{2}}
    \left(\frac{p}{3}\right)
    = \left(\frac{p}{3}\right).
\end{equation*}
\end{remark}
In conclusion, we have proved that
\begin{equation}\label{eqN0}
    N^*_p(0) = 
   \begin{cases}
       1 & \text{ if $p\equiv 5\pmod 6$}\\[2pt]
       2p-1 & \text{ if $p\equiv 1\pmod 6$}.
   \end{cases}
\end{equation}

\smallskip
\noindent
\textbf{B.}
Now suppose $l\not\equiv 0\pmod p$. 
First let us note that if $a\not\equiv 0\pmod p$, then
\mbox{$N^*_p(l) = N^*_p(a^2l)$} because the correspondence
$(u,v)\leftrightsquigarrow (au,av)$ is a a one-to-one mapping between
the solutions of the congruences counted in the two sets
$\cM^*_p(l)$ and $\cM^*_p(a^2l)$.
It follows that $N^*_p(l)$ remains constant for all quadratic residues $l$ and still, $N^*_p(l)$ also is constant for all
non-residues $l$. 
% Then, when $\legendre{l}=1$, we will use $R$ to denote the 
% cardinal numbers $N_p^*(l)$, and $N$ instead
% when $\legendre{l}=-1$.
We denote by $R$ and $N$ the cardinalities in the two cases:
\begin{equation}\label{eqRN}
  \begin{split}
    R := N^*_p(l) \quad\text{ if $\legendre{l}=1$}\,,
    \quad\text{ and } \quad
    N := N^*_p(l) \quad\text{ if $\legendre{l}=-1$}\,.
  \end{split}
\end{equation}

Next, to find $R$, we may assume that $l=1$.
Let us introduce a parametrization of the curve that defines 
the congruence in $\cM_p^*(1)$, intersecting it with all 
possible lines that pass through the base solution $(1,0)$.
This will determine a second solution, which satisfies the system
\begin{equation}\label{eqSystemCL}
    \begin{cases}
        u^2+3v^2\equiv 1\pmod p\\[4pt]
        r(u-1) \equiv v\pmod p
    \end{cases}
\end{equation}
for $r=0,1,\dots,p-1$.
In addition to the lines in~\eqref{eqSystemCL}, the vertical
line $u=1$ gives only the base solution $(1,0)$.
Note that all solutions in $\cM_p^*(1)$ are obtained in 
this way, and to find their exact number we need to check 
for any repetitions.

\smallskip
If $r=0$, then $v=0$, which gives two solutions 
$(1,0)$ and $(p-1,0)$.

If $r\in\{1,\dots,p-1\}$, then~\eqref{eqSystemCL} implies
$u^2+3r^2(u-1)^2\equiv 1\pmod p$, that is,
\begin{equation}\label{eqSystemCL2}
    % \begin{cases}
       (3r^2+1)u^2 -6r^2u +(3r^2-1)\equiv 0\pmod p\,.
    % \end{cases}
\end{equation}

Besides the base solution $u=1$, congruence~\eqref{eqSystemCL2}
has the solution
$u=(3r^2-1)(3r^2+1)^{-1}$, 
provided that $3r^2+1\not \equiv 0\pmod p$.

Let us check if these solutions $u=u(r)$ are distinct for different $r$'s. We see that the equality 
$(3r_1^2-1)(3r_1^2+1)^{-1} = (3r_2^2-1)(3r_2^2+1)^{-1}$,
for distinct $r_1,r_2\in\{0,1,\dots,p-1\}$,
is equivalent with $(r_1 -  r_2)(r_1+r_2)\equiv 0 \pmod p$.
Then, in the only case with uncertainty regarding the possible 
coincidence of solutions, that is when $r_2=p-r_1$, 
we see that the corresponding two $v$'s, that is, $r_1(u-1)$
and $r_2(u-1)$ are distinct because $p\neq 2$, while the case 
$u=1$ was settled before.

If $3r^2+1 \equiv 0\pmod p$, then congruence~\eqref{eqSystemCL2}
reduces to $2u-2\equiv 0\pmod p$, so that system~\eqref{eqSystemCL} 
also has a unique solution, the same $(1,0)$.
This counts only for just two values of $r$ that are
square roots of $(-3)^{-1}$, and it occurs, according to 
Remark~\ref{RemarkEulerQRL}, 
only for \mbox{$p\equiv 1\pmod 6$}.
Thus, checking all $r\in\{0,1,\dots,p-1\}$, 
we find that if $p\equiv 1\pmod 6$, 
the base solution $(1,0)$
appears repeated three times, when $r=0, r_1, r_2$, where
and $r_1$ and $r_2$ are the square roots mod $p$ of 
$(-3)^{-1}$,
and the base solution $(1,0)$ is never repeated otherwise,  
when $(-3)^{-1}$ is not a quadratic residue.

In conclusion, 
% if  $3r^2+1\not \equiv 0\pmod p$,
we have shown that system~\eqref{eqSystemCL}
has a single solution for $r\in\{1,2,\dots,p-1\}$
and two solutions if $r\equiv 0\pmod p$, 
all together, being in a total of $2+(p-1)=p+1$ distinct 
solutions, provided that $p\not\equiv 1\pmod 6$,
and two less, because of the noted repetitions, otherwise.
Therefore:
\begin{equation}\label{eqR}
    R=N^*_p(1)=
    \begin{cases}
        p-1 & \text{ if $p\equiv 1\pmod 6$}\\[4pt]
        p+1 & \text{ if $p\equiv 5\pmod 6$}.
    \end{cases}    
\end{equation}

To complete the analysis, we still need to treat the case 
when $l$ is a quadratic non-residue, that is, to find $N$ defined by~\eqref{eqRN}.
This can be done by reckoning that $\FF_p^2$ is partitioned
into subsets $\cM_p^*(l)$ grouped by their equal 
cardinalities, as follows:
\begin{equation}\label{eqReckon}
    \FF_p^2 = \cM_p^*(0) \cup
    \bigcup_{\substack{l=1\\ \legendre{l}=1}}^{p-1}\cM_p^*(l)
    \cup
    \bigcup_{\substack{l=1\\ \legendre{l}=-1}}^{p-1}\cM_p^*(l).
\end{equation}
Then, on combining~\eqref{eqN0}, \eqref{eqRN}, 
and the fact that there are an equal number of $(p-1)/2$ of 
quadratic residues and non-residues, from~\eqref{eqReckon}
we find that
\begin{equation*}%\label{eqReckon}
    p^2 =  
  \begin{cases}
      2p-1+ \frac{p-1}{2}R + \frac{p-1}{2}N  & \text{ if $p\equiv 1\pmod 6$}\\[6pt]
     1 + \frac{p-1}{2}R + \frac{p-1}{2}N  & \text{ if $p\equiv 5\pmod 6$}.
  \end{cases}
\end{equation*}
On employing~\eqref{eqR}, it follows that
\begin{equation}\label{eqN}
    N = 
  \begin{cases}
       p-1  & \text{ if $p\equiv 1\pmod 6$}\\
       p+1 & \text{ if $p\equiv 5\pmod 6$},
  \end{cases}
\end{equation}
which, compared to~\eqref{eqR}, means that
$N=R$ for all $p$.

To complete the proof of Theorem~\ref{TheoremC} 
for the germs $(0,1,1)$, $(1,0,1)$ and $(1,1,0)$, we
only have to replace the 
values of $N^*_p(l)$, explicitly obtained in relations
\eqref{eqN0}, \eqref{eqRN}, \eqref{eqR} and~\eqref{eqN}, in the definition $\dens_p(l)=N_p^*(l)/p^2$, for
$p\ge 5$ and $0\le l\le p-1$, to obtain the formulas of the densities from the statement of the theorem.

%%%%%%%%%%%%%%%%%%%%%%%%%%%%%%%%%%%%%%%
\subsection{The germ \texorpdfstring{$(0,0,0)$}{(0,0,0)}}
Let
\begin{equation}\label{eqNpl}
	N_p(l) = \#\big\{(x,y) : x^2+y^2+xy-x-y \equiv l \pmod p\big\}.
\end{equation}
For any $l$, $0\le l \le p-1$, as seen above, 
the limit density is $\dens_p(l) = N_p(l)/p^2$.
\medskip

\noindent
\texttt{Case $p=2$}.
With $F(x,y)=x^2+y^2+xy-x-y$, and $x,y\in\{0,1\},$
the values modulo~$2$, are $0, 0, 0, 1$, so that 
$\dens_2(0) = 3/4$ and $\dens_2(1) = 1/4$.  
% 0  0  
% 0  1 
\medskip

\noindent
\texttt{Case $p=3$}.
The values of $F(x,y) \pmod 3$, for $y=0,1,2$, are: $0,0,2$ if $y=0$;
$0,1,1$ if $y=1$; and $2,1,2$ if $y=2$.
% Therefore 
Then
\mbox{$\dens_3(0)=\dens_3(1)=\dens_3(2)=3/3^2=1/3$.}
% 0  0  2  
% 0  1  1  
% 2  1  2  

\medskip

\noindent
\texttt{Case $p\ge 5$}.
As before, we change the variables to transform the expression from the left side of the congruence into a canonical quadratic form. First, we multiply the congruence by $4$ and complete the square to get rid of the $xy$ term. Then replacing $x$ by $u$, where $u=2x+y$, the congruence inside~\eqref{eqNpl} becomes 
\begin{equation*}%\label{equy01}}
	u^2+3y^2-2u-2y \equiv 4l \pmod p\,.
\end{equation*}
Next we complete the square to get rid of the linear term $-2u$. For this, we replace $u$ by~$w$, where $w=u-1$, and the above congruence becomes
 \begin{equation*}%\label{equy02}}
	w^2+3y^2 -2y \equiv 4l+1 \pmod p\,.
\end{equation*}
Now we eliminate the remaining linear term by multiplying the congruence 
by $3$ and completing the square. We replace $y$ by the new variable
$v$, where $v=3y-1$, and arrive at the new form of the congruence
 \begin{equation*}%\label{equy03}}
	3w^2+v^2 \equiv 4(3l+1) \pmod p\,.
\end{equation*}
In the end, we can go back to the original variables by putting
$x=2^{-1}v$ and $y = 2^{-1}w$.
Considering the fact that $p\neq 2,3$, the solutions of the 
several congruences above are in a one-to-one correspondence with each 
other, the number of solutions counted by $N_p(l)$ in~\eqref{eqNpl}
equals $N^*_p(3l+1)$, the number of solutions of the last new congruence, 
where
\begin{equation}\label{eqNp3l}
    N^*_p(3l+1) = \#\big\{(x,y) : x^2+3y^2 \equiv 3l+1 \pmod p\big\}.
\end{equation}
This is the same as in the case of germ $(0,1,1)$, except that $l$ is replaced by $3l+1$.
Adapting the analysis from there, the cases are distinguished 
depending on whether $3l+1\equiv 0\pmod p$ or $3l+1\not\equiv 0\pmod p$.

Note that $(-3)^{-1}=\frac{p-1}{3}$ if $p\equiv 1 \pmod 6$
and $(-3)^{-1}=\frac{2p-1}{3}$ if $p\equiv 5 \pmod 6$.

\medskip

If $3l+1\equiv 0\pmod p$, then the number of solution of the congruence
in~\eqref{eqNp3l}, which becomes $ x^2\equiv -3y^2 \pmod p$, 
depends on whether $-3$ is a quadratic residue or not modulo~$p$.

If $\legendre{-3}=-1$, then there is only one solution at $(x,y)=(0,0)$,
so that $\dens_p\big(\frac{p-1}{3}\big)=\frac{1}{p^2}$
or $\dens_p\big(\frac{2p-1}{3}\big)=\frac{1}{p^2}$, 
depending on the case where $p\equiv 1 \pmod 6$ or $p\equiv 5 \pmod 6$.

If $\legendre{-3}=1$, then for each $x\in\{1,\dots, p-1\}$,
there are two distinct solutions $(x,\pm y)$ for $y\in\{1,\dots, p-1\}$,
plus one more $(0,0)$, resulting in a total of $2p-1$ solutions.
Therefore, 
$\dens_p\big(\frac{p-1}{3}\big)=\frac{2p-1}{p^2}$ or
$\dens_p\big(\frac{2p-1}{3}\big)=\frac{2p-1}{p^2}$, 
depending on the case where $p\equiv 1 \pmod 6$ or $p\equiv 5 \pmod 6$.

It remains to be seen if this is consistent with the statement of the theorem, and this can be checked because condition
$\legendre{-3}=1$ holds if and only of $3$ divides $p-1$
(which reduces in our case to $p\equiv 1\pmod 6$).
The last equivalence follow 
by combining Euler's criterion with the quadratic reciprocity because 
law, because
\begin{equation*}
    \legendre{-3} = \legendre{-1}\legendre{3} 
    = (-1)^{\frac{p-1}{2}}\cdot (-1)^{\frac{p-1}{2}\cdot \frac{3-1}{2}}
    \left(\frac{p}{3}\right)
    = \left(\frac{p}{3}\right).
\end{equation*}
This conclude the proof of Theorem~\ref{TheoremC}.

%%%%%%%%%%%%%%%%%%%%%%%%%%%%%%%%%%%%%%%%%%%%%%%%%%%%
\section{The number tiling of the plane and the proof of Theorem~\ref{TheoremA}}\label{SectionRuleD}
The aim of this section is to prove Theorem~\ref{TheoremA}.
To do so, we will begin by offering a more detailed account of the tiling coverage of the plane with integers,
which will lead to a parameterization of all nodes within the triangular network $\cT_H(a,b,c)$.

\subsection{A parametrization of the represented integers}\label{SectionParametrization}
Let $a,b,c$ be fixed integers.
In order to determine which elements belong to the sequence of integers represented 
in $\cR_H(a,b,c)$, we need to describe the structure of the set of triples
$\cT_H(a,b,c)$. 
In this set, although the components of the triples appear multiple times in different 
triples, one can track the step-by-step generation of triples starting from 
the basic germ $(a,b,c)$.

We start building an interlinked network of numbers based at the nodes of a~$2$-dimensional  
triangular network by placing $a,b,c$ in the nodes of one of 
the smallest equilateral triangles of the grid. This is the the base triangle 
that stores the germ $(a,b,c)$. We will identify and then refer to the nodes by the 
integers they contain.
The integers in the nodes will be called \textit{weights.}
Then, we place the weight 
$d = H'(a,b,c) = -a + b+c+1$ in the fourth corner of the lozenge $abdc$ with the shorter
diagonal $bc$.

%%%%%%%%%%%%%%%%%%%%%%%%%%%%%%%%%%%%%%%%%%%%%%%%%%%%%%%%%%%%%%%%%%%%
% https://sage.syzygy.ca/jupyter/user/sucodru/notebooks/SQUARE_PRIMES/Hexagon%20Filling%20Fibonacci.ipynb
% some manual editing, i.e., scale, rotation, inclusion of the colorbar
% \setlength{\belowcaptionskip}{-10pt}
\begin{figure}[ht]
 \centering
 \includegraphics[width=0.35\textwidth]{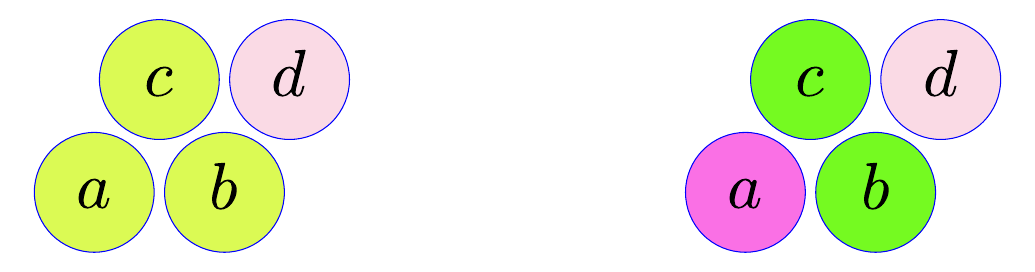}
\caption{The lozenge expansion of $d$ from $a$ across
the short diagonal~$bc$.
}
 \label{FigRuleD}
 \end{figure}
The process of expanding towards a future node from a previous one, by placing the child-integer 
symmetrically across a side of a triangle, as shown in Figure~\ref{FigRuleD}, 
is then iterated to generate new integers
across any side of the triangles in the current generation. 
The expansion can be done in any order using either of 
the corresponding
operators $H',H'',H'''$, with some rule, such as 
to produce a certain linear development in a given direction, 
randomly, or in a circular fashion, as shown by the 
individual steps from Figure~\ref{FigConsistencyHex}. 
(The circular spiral construction is also employed 
in generating the first two circular annuli around 
the base germs $(0,0,0)$ and $(0,1,1)$ in 
Figure~\ref{FigureSpirals}.)
%%%%%%%%%%%%%%%%%%%%%%%%%%%%%%%%%%%%%%%%%%%%%%%%%%%%%%%%%%%%%%%%%%%%
% https://sage.syzygy.ca/jupyter/user/sucodru/notebooks/SQUARE_PRIMES/Hexagon%20Filling%20Fibonacci.ipynb
% some manual editing, i.e., scale, rotation, inclusion of the colorbar
\begin{figure}[ht]
 \centering
 \includegraphics[width=0.99\textwidth]{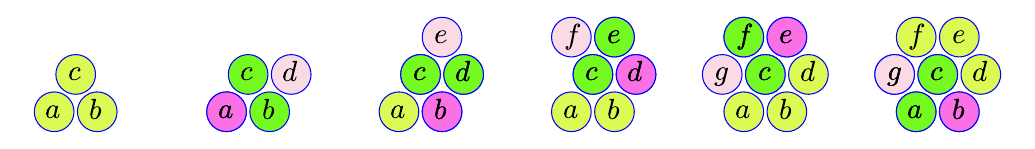}
\caption{Consistency check on the immediately adjacent hexagonal path around the initial triangle $abc$. 
One finds the same value of $g$ regardless of whether one goes around in trigonometric direction in four
steps or vice versa in just one step.
}
 \label{FigConsistencyHex}
 \end{figure}

We need to ensure that this procedure can be repeated indefinitely to fill 
all the nodes in the triangular network covering the plane. 
The construction must be consistent, meaning that the same weights are obtained 
regardless of the path taken to reach the nodes. 
To verify this, we will check the correctness on the minimal paths around a node 
and use the fact that $H', H''$ and $H'''$ are involutions (as we know from Proposition~\ref{PropositionInvolutionNonComute}), which allows to go back and forth on the paths. 
This will guarantee a unique and consistent filling of the nodes of the entire triangular tiling, 
as we will then have established step by step the uniqueness 
of the weights on any path between two fixed nodes.

In counterclockwise direction, the weights in Figure~\ref{FigConsistencyHex} are:
\begin{equation}\label{eqg1}
\begin{split}
    d & = H'(a,b,c) = -a + b+c+1\\ %(d,b,c)
    e & = H''(d,b,c) = -b + d+c+1
	%-b + (-a + b+c+1) +c +1
	= -a +2c +2 \\ %(d,e,c)
     f & = H'(d,e,c) = -d + e+c+1
	% = -(-a + b+c+1) + (-a +2c +2)+ c+1
	= -b +2c +2 \\ %(f,e,c)
    g & = H''(f,e,c) = -e + f+c+1
	% = -(-a +2c +2) + (-b +2c +2)+c+1
	= a-b +c +1, \\ %(f,g,c)
\end{split} 
\end{equation}
and then clockwise:
\begin{equation}\label{eqg2}
    \begin{split}
    g & = H''(a,b,c) = -b + a+c+1\,.
    \end{split} 
\end{equation}
Since the same weight $g$ is produced on both paths~\eqref{eqg1} and~\eqref{eqg2},
we have proven the consistency of the construction.

\begin{remark}\label{Remark1}
   The symmetries in relations~\eqref{eqHHH} and~\eqref{eqg1} imply:
\begin{enumerate}
    \item    
The symmetric correspondence of the formulas of $e$ and $f$ in~\eqref{eqg1} 
reflect the position of nodes $a$ and $b$ in Figure~\ref{FigConsistencyHex}.
    \item    
Let $x,y,z$ be the weights of three consecutive adjacent nodes in the triangular network. 
Then  $z=-x+2y+2$.
\end{enumerate}
\end{remark}

Iterating the second part of Remark~\ref{Remark1}, it follows
that the weights that follow $x$ and $y$ in straight line are:
$-x+2y+2$, $-2x+3y+6$, 
  $-3x+4y+12$,  $-4x+5y+20$, and so on.
Then, by induction, we obtain the following result.
%%%%%%%%%%%%%%%%%%%%%%%%%%%%%%%%%%%%%%%%%%%%%%%%%%%%%%%%
\begin{lemma}\label{LemmaSL}
    The weights of the nodes on the straight line determined by 
the adjacent nodes $x$ and~$y$ are 
\begin{equation}\label{eqFSL}
        L(x,y \mid k) = -(k-1)x +ky+k(k-1)\ \ \text{ for $k\ge 0$.}    
\end{equation}    
\end{lemma}

Now, just using the one-dimensional parametrization~\eqref{eqFSL}, one can prove 
Theorem~\ref{TheoremA}. For this it suffices to check the weights on the sets 
of parallel half-lines on the six directions rotated by $60$ degrees around any fixed 
origin, and note that the size of the quadratic term $k^2-k$ allows for at most 
a finite number of negative weights around the origin.

A complete two-dimensional parametrization of the weights is given in the next section.

%%%%%%%%%%%%%%%%%%%%%%%%%%%%%%%%%%%%%%%%%%%%%%%%%%%%%%%%%%%%%%%%%%%%%%%%%%%%%%%%%
\subsection{Two-dimensional parametrization of the weights}\label{subsection2D}
%%%%%%%%%%%%%%%%%%%%%%%%%%%%%%%%%%%%%%
% \section{The closed formula for all the nodes of the network}

Let $\varrho = 1/2+\sqrt{3}i/2$
% $\varrho = \sdfrac 12+\sdfrac{\sqrt{3}}{2}i$ 
be a root of order $6$ of $1$. 
% Note that
% $\varrho$ is also a root of order $3$ of~$-1$ and it satisfies the `dropping level' equation 
% \begin{equation}\label{eqrho}
%    \varrho^2 = \varrho -1\,.
% \end{equation}
Any node of the triangular lattice can then be uniquely expressed using the base $\{1,\varrho\}$ 
as the linear combination $m+n\varrho$ with $m,n\in\ZZ$. Thus, the nodes of the network are 
the elements of the set
\begin{equation*}
   \begin{split}
   \cH :=    \{m+n\varrho : m,n\in\ZZ\}\,.
   \end{split}   
\end{equation*}
Let $G(a,b,c \mid m,n)$ denote the value of the weight associated to the node $m+n\varrho$ 
in the triangular network generated by the base triangle of weights $a$,$b$, and $c$.

To compare, it should be noted that $G$ 
must be two-dimensional and dependent on $c$, while 
$L$ is one-dimensional and independent of $c$, and they are related by
\begin{equation*}
    L(a,b \mid m) = G(a,b,c \mid m,0) \ \ 
    \text{ for any $a,b,c,m\in\ZZ$.}
\end{equation*}

The next theorem provides the %general 
closed-form expression of \mbox{$G(a,b,c \mid m,n)$.}
%%%%%%%%%%%%%%%%%%%%%%%%%%%%%%%%%%%%%%%%%%%%%%%%%%%%%%%%%%%%%%%%%%%%
\begin{theorem}\label{LemmaDL}
   For any base triple $(a,b,c)$ and any integers $m,n$, we have
\begin{equation}\label{eqFmn}
   G(a,b,c\mid m,n) = -(m+n-1)a  + mb + nc + (m^2+ n^2 +mn- m -n)\,.    
\end{equation}
\end{theorem}
\begin{proof}
Let $m,n\in\ZZ$ and let $a,b,c$ be fixed real numbers.
While the linear expansion in Lemma~\ref{LemmaSL} uses the weight $c$ just as a support to
generate the weights on the line $ab$, here we generate and also record all the weights 
on the line that starts at $c$ and is adjacent to~$ab$. 
Then, iterating the process, the general formula will follow.
% https://sage.syzygy.ca/jupyter/user/sucodru/notebooks/SQUARE_PRIMES/Hexagon%20Filling%20Fibonacci.ipynb
% some manual editing, i.e., scale, rotation, inclusion of the colorbar
\begin{figure}[ht]
 \centering
    \includegraphics[width=0.69\textwidth]{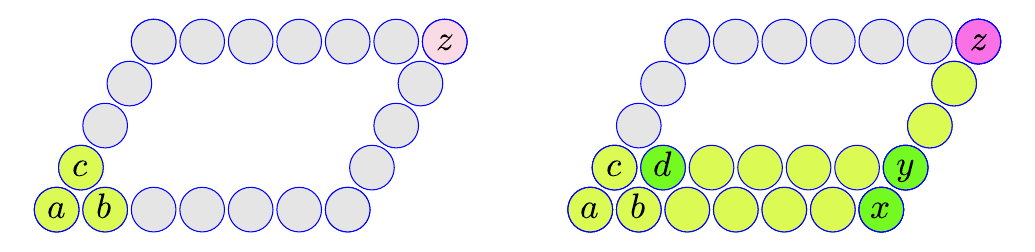}
\caption{The \emph{parallelogram rule} used to obtain the closed formula for 
$G(a,b,c \mid m,n)$. In order to find $z$ from the opposite corner of $a$ 
and having the coordinates $m,n$ in base $\{1,\varrho\}$, 
one applies Lemma~\ref{LemmaSL} a few times to find $x$, 
the $m$-th element on the line $ab$, 
then $d$, with $d=H'(a,b,c)$, next $y$, the $m$-th element on the line $cd$, 
and finally one arrives at $z$, the $n$-th element on line~$xy$.
}
 \label{FigGeneralFormula}
 \end{figure}
 
Let $z=G(a,b,c \mid m,n)$ be a short notation of the unknown variable 
for the fixed variables $a,b,c,m,n$.
Next we use Figure~\ref{FigGeneralFormula} as a support in the following successive steps:
\begin{enumerate}
   \item 
   By~\eqref{eqFSL} we find $x=L(a,b \mid m)$.
      \item 
    Then $d = H'(a,b,c) = -a+ b+c +1$.
         \item 
    Next, by~\eqref{eqFSL} we find $y=L(c,d \mid m)$.
            \item 
   Finally, by~\eqref{eqFSL} again we find $z=L(x,y \mid n)$.
\end{enumerate}
Then writing $x,y$, and $z$ in terms of $a,b,c$, we find
\begin{equation*}
   \begin{split}
   x =&  -(m-1)a + mb + m(m-1),\\
   y =&  -(m-1)c + md + m(m-1)
%    = -(m-1)c + m(-a+ b+c +r) + m(m-1)r\\
   =  -ma +mb +c + m^2,\\
   \end{split}   
\end{equation*}
so that
\begin{equation*}
   \begin{split}
   z =&  -(n-1)x + ny + n(n-1) \\%=\\
%    = &-(n-1)\big(-(m-1)a + mb + m(m-1)r\big) + n\big(-(m-1)c + md + m(m-1)r\big) + n(n-1)r
%       = m^2r + mnr + n^2r - am + bm - an + cn - mr - nr + a
     =& -(m+n-1)a  + mb + nc + (m^2+ n^2 +mn- m -n)\,.
%      -(m+n-1)a  + mb + nc + (m^2+ n^2 +mn- m -n)*r
   \end{split}   
\end{equation*}
The last line gives us the needed formula for $z=G(a,b,c \mid m,n)$, which completes the proof of the theorem.
\end{proof}

The operators $H',H''$, and $H'''$ are not linear because of the constant $1$ that is added in their definition.
However, by subtracting component-wise the triples obtained as images of 
triples with one equal component on the same position obtained through a composition of a finite sequence of 
$H', H''$ and $H'''$, the constant term cancels. As a result we obtain a linear function with integer coefficients 
that has a zero at a lattice point.
Consequently, it will have infinitely many zeros at other lattice points equally spaced on a straight line, 
as noted in the example before Remark~\ref{Remark11}.
Alternatively, the same result can be achieved by using the complete characterization 
of the weights from Theorem~\ref{LemmaDL}.

%%%%%%%%%%%%%%%%%%%%%%%%%%%%%%%%%%%%%%%%%%%%%%%%%%%%%%%%%%%%%%%%%%%%
% https://sage.syzygy.ca/jupyter/user/sucodru/notebooks/SQUARE_PRIMES/Hexagon%20Filling%20Fibonacci.ipynb
% some manual editing, i.e., scale, rotation, inclusion of the colorbar
\setlength{\belowcaptionskip}{-10pt}
\begin{figure}[ht]
 \centering
  \includegraphics[width=0.85\textwidth]{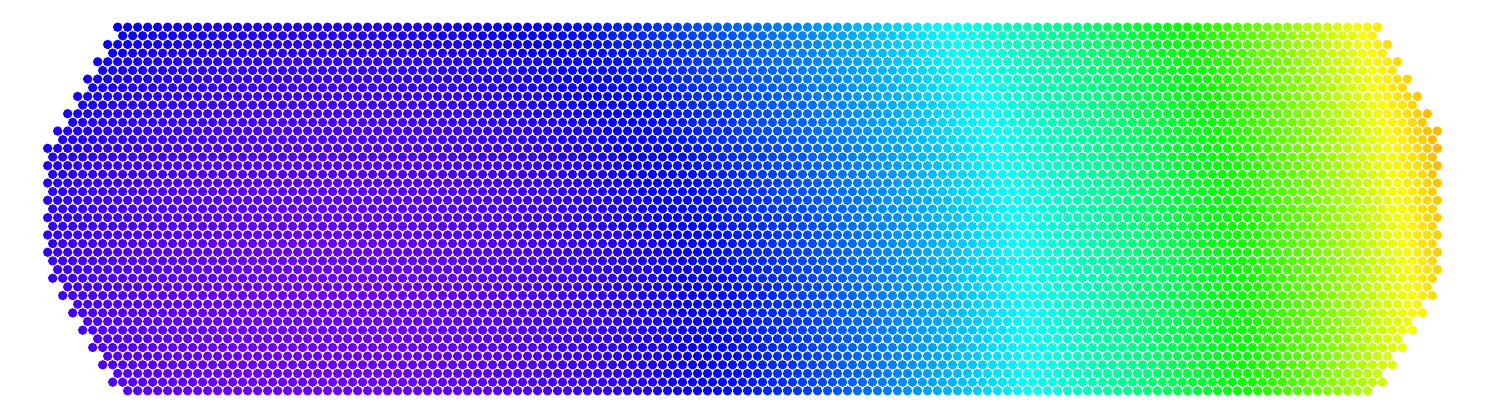}
%  Max(a)= 13699
% Number of nodes = 5801 out of all possible 16222
%  \includegraphics[width=0.75\textwidth]{BandSize22000Nodes11183Max10929Initial202321092092hsv_r.pdf}
%  Number of nodes = 4631 out of all possible 11183
%   \includegraphics[width=0.65\textwidth]{HexSize8000Nodes4111Max6537Initial202321092092.pdf}
\caption{
The spectrum of colors corresponding to the weights of $5801$ nodes generated by~\eqref{eqHHH} in a rectangular neighborhood of the initial triangle with weights $2023,2109,2092$.
% The spectrum of colors corresponding to the weights of $4111$ nodes generated by~\eqref{eqRuleD}
% with $r=1$ around the initial triangle $2023,2109,2092$.
} 
 \label{FigSpectrum}
 \end{figure}
%%%%%%%%%%%%%%%%%%%%%%%%%%%%%%%%%%%%%%%%%%%%%%%%%%%%%%%%%%%%
\subsection{Final notes}\label{subsectionSpectrum}
Except for a few special cases, regardless of the initial triangle, the size of the numbers 
generated by~\eqref{eqHHH}, shows a wide spectrum of nuances.
A colored representation of the weights along and across the plane and a few explicit values of 
the weights are shown in Figure~\ref{FigSpectrum} and Figure~\ref{Fig2Examples}, respectively.
Two examples of an extract of the network of weights generated by the initial triangles $4,7,5$ and 
$2023,2109,2092$, respectively, are shown in Figure~\ref{Fig2Examples}.
It should be noted that although the numbers that generate the spectrum are discrete, 
it has a continuum appearance without a distinct break separating the colors.
Also, the `visual spectrum' is similar all over the plane, 
regardless of the area where one makes an analogous representation.
%  \noindent
% 
% HexSiz%%%%%%%%%%%%%%%%%%%%%%%%%%%%%%%%%%%%%%%%%%%%%%%%%%%%%%%%%%%%%%%%%%%%
% https://sage.syzygy.ca/jupyter/user/sucodru/notebooks/SQUARE_PRIMES/Hexagon%20Filling%20Fibonacci.ipynb
% some manual editing, i.e., scale, rotation, inclusion of the colorbar
\begin{figure}[ht]
%  \centering
%  \mbox{
%  \subfigure{
    \includegraphics[width=0.49\textwidth]{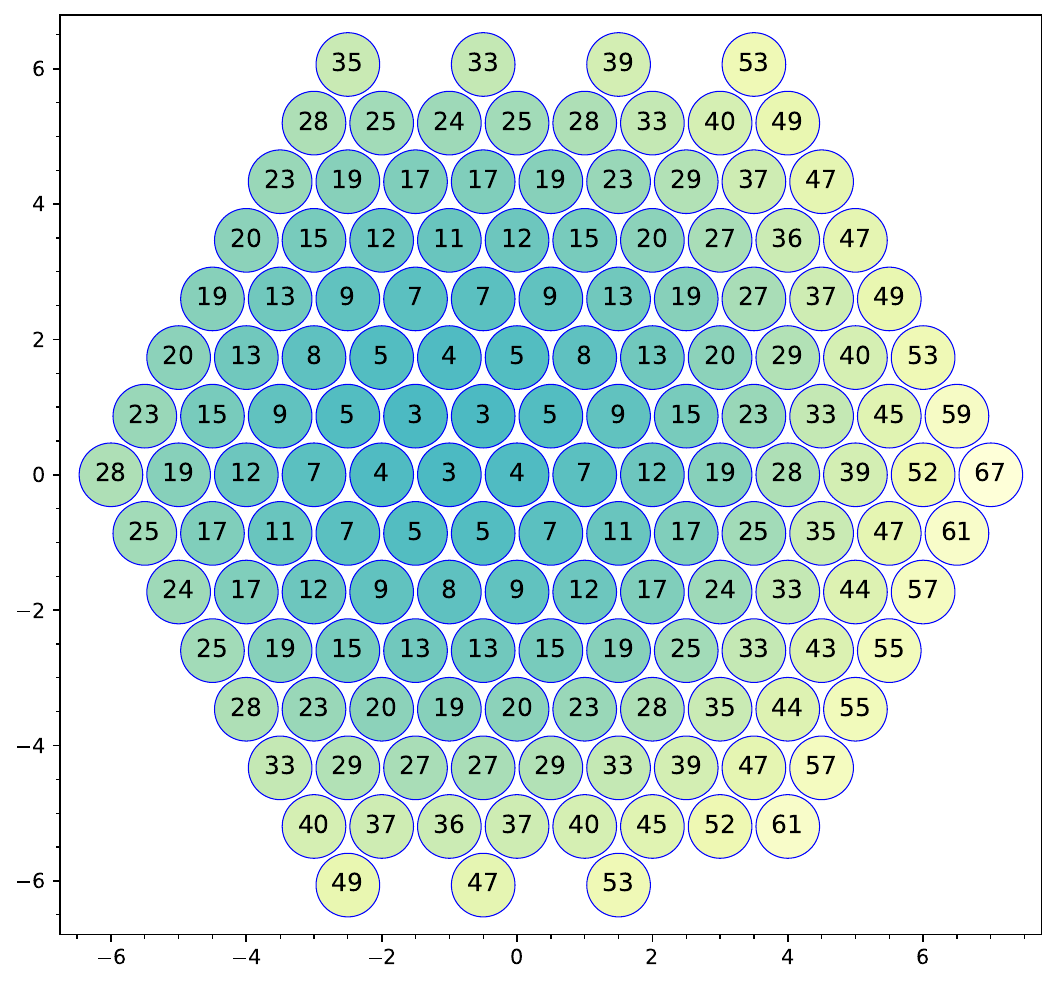}
%  \label{Fig2Examples1}
%  }%\quad
%  \subfigure{
    \includegraphics[width=0.476\textwidth]{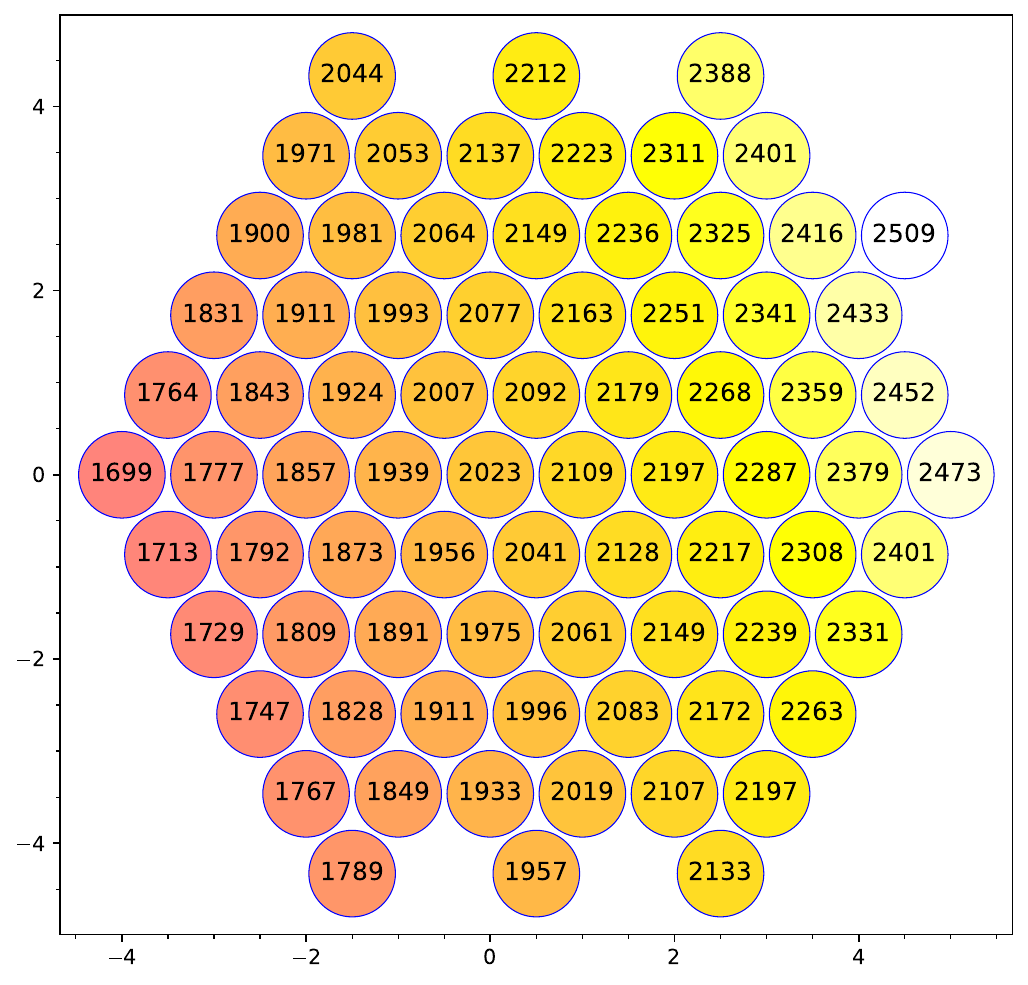}
%  \label{Fig2Examples2}
%  }
%  }
\caption{Triangular lattices filled recursively in the nodes by weights given by 
relations~\eqref{eqHHH} across any short diagonal of a lozenge.
The image on the left shows $167$ nodes whose maximum value is $67$, all of which are generated 
by the initial triangle with weights of $4,5,7$.
In the image on the right there are shown only $77$ nodes with weights between $1699$ and $2509$.
Color intensity shades indicate the size of the numbers in the nodes to distinguish them from each other.
}
 \label{Fig2Examples}
 \end{figure}

\subsubsection{Evaluation of the number of negative weights}\label{subsectionEvaluationNegativeWeights}
Let us find the path from $(0,0,c)$, where $c\ge 0$, to the minimum triple in the covering that contains $(0,0,c)$.
Along the way we will meet three remarkable sequences and also find the tower to which $(0,0,c)$ belongs to. 

Let $H$ be the composition operator defined by $H:=H'\circ H''\circ H'\circ H'''$.
A direct calculation shows that, for any triple $(a,b,c)$ with $a=b$, we have
\begin{equation*}%\label{eqH2}
    H(a,a,c) = \big( 3a-2c+4, 3a-2c+4, 2a-c+1  \big)\,.
\end{equation*}

This is a zigzag operator whose iterations define the shortest path in the triangular network 
from $(0,0,c)$ towards the center. (The \textit{center} is the triple in the network
whose components are the smallest and closest to each other.)
Iterating the application of $H$, by induction we find that 
\begin{equation}\label{eqH2}
    \begin{split}
     H^{[n]}(a,a,c) = \big(& (2n+1)a-2nc+n(3n+1),\\ 
     &(2n+1)a-2nc+n(3n+1),\ \\ 
     &2na-(2n-1)c+n(3n-2)  \big)   
    \end{split}
\end{equation}
for $n\ge 1$.

\begin{remark}\label{RemarkRemarkableSequences}
Formula\eqref{eqH2} employs two remarkable sequences if $a=c=0$.

The first one is $4, 14, 30$, $52, 80$, $114$, $154$, $200$, $252$, 
$310,\dots$ and
the formula for the $n$-th element is $\mathtt{NE}(n):= n(3n+1)$ for $n\ge 1$.
It is the so-called sequence of the
\textit{twice second pentagonal numbers}~\cite[{\href{https://oeis.org/A049451}{A049451}}]{oeis}.

The second sequence is that of \textit{octagonal} or the 
\textit{star numbers}~\cite[{\href{https://oeis.org/A000567}{A000567}}]{oeis},
with elements $1, 8, 21, 40, 65, 96, 133, 176, 225, 280,\dots$,
and the formula for the $n$-th element given by $\mathtt{SW}(n) := n(3n-2)$
for $n\ge 1$ (see~\cite{CP2023} for the occurrence of this sequence in the context of some polytopal conjectures for Coxeter groups). 

We note that creating a spiral pattern by arranging the natural numbers,
starting from the center and moving towards the bottom left, and then 
in a thread-like manner around an ever-expanding hexagon,
$\mathtt{NE}(n)$ is the sequence of numbers on the radius formed from the center 
towards the North-East corner, while 
$\mathtt{SW}(n)$ represents the sequence formed on the opposite radius 
starting from the center towards the South-West corner.
\end{remark}
% \medskip
\noindent
\texttt{Case $c\equiv 0\pmod 3$.}
Starting with $(0,0,c)$, the iterations~\eqref{eqH2} lead to a triple
of the form $(s,s,s)$. For this to happen, the necessary number of steps can be deduced from the 
equality
\begin{equation*}
    -2nc+n(3n+1) = -(2n-1)c+n(3n-2),
\end{equation*}
which implies $n = \frac c3$. Then the triple with all components equal is
\begin{equation}\label{eqMin-c0}
    \left( -\sdfrac{1}{3}(c^2-c),
    -\sdfrac{1}{3}(c^2-c),
    -\sdfrac{1}{3}(c^2-c)
     \right).
\end{equation}
Translating up this triple by $\sdfrac{1}{3}(c^2-c)$, we arrive at the germ
$(0,0,0)$. 
Employing Theorem~\ref{TheoremGerms}, it follows that if $c\equiv 0 \pmod 3$, then
$-\sdfrac{1}{3}(c^2-c)$ is the minimum weight in a node of the triangular network that contains the
triple $(0,0,c)$.

\medskip
\noindent
\texttt{Case $c\equiv 1\pmod 3$.}
Starting from $(0,0,c)$ and applying the iterations of $H$, 
we arrive next to the center of the triangular network at a triple of the form
$(s,s,s+1)$. Employing formula~\eqref{eqH2}, we see that this occurs
if
\begin{equation*}
    -2nc+n(3n+1) + 1 = -(2n-1)c+n(3n-2),
\end{equation*}
that is, $n = \frac{c-1}{3}$.
Then, applying the operator $H'''$ to this last triple, we arrive at the 
triple with all components equal, triple given by the same formula in~\eqref{eqMin-c0}.
% \begin{equation}\label{eqMin-c0}
%     \left( -\sdfrac{1}{3}(c^2-c),
%     -\sdfrac{1}{3}(c^2-c),
%     -\sdfrac{1}{3}(c^2-c)
%      \right).
% \end{equation}
Then, as in the previous case, translating the triple up by $\sdfrac{1}{3}(c^2-c)$, we arrive at the germ
$(0,0,0)$, so that, by Theorem~\ref{TheoremGerms}, it follows that if $c\equiv 1 \pmod 3$, then
$-\sdfrac{1}{3}(c^2-c)$ is the minimum weight in a node of the triangular network that contains the
triple $(0,0,c)$.

\medskip
\noindent
\texttt{Case $c\equiv 2\pmod 3$.}
Proceeding as in the previous cases, one finds that the first triple that is 
close to the center and is in the image of the iterations of $H$
has the form $(s,s,s+2)$. 
% We will find this triple and see that this claim holds true.
In order for this to happen, we need to take $n = \frac{c-2}{3}$.
Then, applying $H'''$ on the triple in~\eqref{eqH2} with  $n = \frac{c-2}{3}$, we obtain
\begin{equation}\label{eqMin-c2}
    \left( -\sdfrac{1}{3}(c^2-c - 2),
    -\sdfrac{1}{3}(c^2-c - 2),
    -\sdfrac{1}{3}(c^2-c + 1)
     \right),
\end{equation}
triple which translated up by $\sdfrac{1}{3}(c^2-c + 1)$
is seen to belong to the tower of germ $(1,1,0)$. 

In conclusion, by Theorem~\ref{TheoremGerms} it follows that 
if $c\equiv 2 \pmod 3$, then
$-\sdfrac{1}{3}(c^2-c + 1)$ is the minimum weight in a node of the triangular network that contains the
triple $(0,0,c)$.
  % (0,0,n) --> (1,1,0)
  %   for n = 2,     5,6,11,14,17,...
  %   Formula scrisă cu operatorii inverși, ca la matematică, ultima e translația cu nu știu cât.
  %   002 --H3--> 00-1 --S1--> 110    i.e., S(1)H3(0, 0, 2) = (1, 1, 0)
  %   005 --H3--> 00-4 --H1--> -30-4 --H2--> -3-6-4  --H1--> -6-6-4  --H3--> -6-6-7 --S7--> 110
   
  %   S(1)H3(H1H2H1H3)^0(0, 0, 2) = (1, 1, 0)
  %   S(7)H3(H1H2H1H3)^1(0, 0, 5) = (1, 1, 0)
  %   S(19)H3(H1H2H1H3)^2(0, 0, 8) = (1, 1, 0)
  %   S(37)H3(H1H2H1H3)^3(0, 0, 11) = (1, 1, 0)
  %   S(61)H3(H1H2H1H3)^4(0, 0, 14) = (1, 1, 0)
  %   S(91)H3(H1H2H1H3)^5(0, 0, 17) = (1, 1, 0)

\medskip
Putting together the results in the tree cases, we find that 
the minimum weight of a node in the triangular network that contains
the triple $(0,0,c)$ is $-\left\lfloor\sdfrac{1}{3}(c^2-c+1)\right\rfloor$ for
any $c\ge 1$.
The sequence of the absolute values of the minima is
$0, 1, 2, 4, 7, 10, 14, 19, 24, 30, 37,\dots$, the remarkable 
sequence~\cite[{\href{https://oeis.org/A007980}{A007980}}]{oeis}.
Among the many properties it has (see~\cite{RG2021}), we mention that it 
equals the number of partitions of $2n$ into at most three parts,
and it is
the number of linearly-independent terms at $2n$-th order 
in the power series expansion of a trigonal rotational-energy-surface
that describes the potential energy of a molecule 
as it undergoes rotational motion around a trigonal axis.
% A trigonal rotational-energy-surface refers to 
% the energy landscape of a molecule or system that exhibits trigonal symmetry.
% In other words, it describes the potential energy of the system 
% as it undergoes rotational motion around a trigonal axis.

\begin{remark}\label{RemarkEnd}
    Based on the analysis above, and in accordance with Theorem~\ref{TheoremA},
it is possible to make an accurate estimation and even calculate algorithmically
the number of nodes with negative weights.
Thus, since the considered triples $(0,0,c)$, with $c\ge 0$,
are located in the triangular network they generate on the boundary 
that separates negative weights from positive ones, 
and since the minimum triple is in the center,
it follows that the number of negative weights is 
asymptotically equal to 
$\frac{2\pi}{3\sqrt{3}}c^2$ 
% or $\frac{2\pi}{3}c^2$ or 
% $\frac{4\pi}{3\sqrt{3}}c^2$ 
% or $\frac{\sqrt{3}}{3}c^2$ 
as $c\to\infty$,
at any level, in both towers that contain the considered triples.
\end{remark}
In order to give a sense of scale, we mention that the
minimum weight in a node of the triangular network that contains 
$(0,0,c)$, with $c=100$, is $-3300$, 
and there are precisely $11946$ negative weights in the nodes.
Then, the ratio between the number of negative weights 
and the just mentioned approximation
is, in this case, only
$11946 \cdot \frac{3\sqrt{3}}{2\pi c^2}\approx 0.98793$
instead of the limit $1$, which will be attained as $c$ tends to infinity.

% min = -(c^2-c)/3 = - (10000-100)/3 = 9900/3 = 3300
% aa,bb,cc,SIZE =  0 0 100 110000
% count = 11946
% 4*pi/3/sqrt(3)*c^2 = 24183.9915231229
% 2*pi/3/sqrt(3)*c^2 = 12091.9957615615
% count/good = 0.98793
%%%%%%%%%%%%%%%%%%%%%%%%%%%%%%%%%%%%%%%%%%%%%%%%%%%

\bigskip
% \textbf{\sc Conclusion.}
\subsection*{Conclusion}
We have introduced a three folded operator~\eqref{eqHHH}, which, starting with a 
triple of integers and moving back and forth from any point in any direction,
generates a triangular tessellation of the plane 
with integers in the nodes, called weights. This operator further leads to a natural relation that defines 
as equivalent tessellations obtained from each other through a translation.
As a consequence, we obtain exactly four classes of equivalence of
tessellations organized in four towers, generated in their very beginnings by
$(0,0,0)$, on one hand, and the more closely related
the other three built around $(0,1,1)$, $(1,0,1)$ 
and $(1,1,0)$, on the other hand.
We have obtained a characterization of the weights
(see Theorems~\ref{TheoremA},~\ref{TheoremGerms},~\ref{LemmaDL}) and the densities of their distribution in residue classes modulo a prime number (Theorem~\ref{TheoremC}), 
and along the way, we have encountered remarkable sequences that reveal intricate patterns concealed within the networks.
This has the potential to offer a fresh perspective or open up an alternative approach
to the practical applications in which they appear, including the geographical model known as the
Central Place Theory~\cite{Bat2013, Mar1977}, recently enhanced with a fractal insertion~\cite{AA1989,BGNR2023},
in a molecule or a virus model~\cite[Chapter~3]{Pet1998},
at the intersection between arts and mathematics~\cite{Rus2017},
or in the modern communication systems (see~\cite{KRNG2024} and the references therein).
Additionally, it could be worthwhile to continue studying the potential implications
of the tesselations with integers studied above in relation to the special lozenge tilings~\cite{CF2023, Ciucu2005}
and their implications in two dimensional electrostatics.

% \bibliographystyle{plain}% no references to doi's
% \bibliographystyle{plainnat}

%%%%%%%%%%%%%%%%%%%%%%%%%%%%%%%%%%%%%%%%%%%
%% ArXiv doesn't du bibtex????
%\bibliographystyle{plainurl}% shows urls
%\bibliography{lozengebib}% common bib file

\begin{thebibliography}{10}

\bibitem{AA1989}
Sandra~L. Arlinghaus and William~C. Arlinghaus.
\newblock The fractal theory of {Central Place Geometry: A Diophantine Analysis
  of Fractal Generators} for arbitrary {Löschian} numbers.
\newblock {\em Geographical Analysis}, 21(2):103--121, 1989.
\newblock \href
  {https://doi.org/https://doi.org/10.1111/j.1538-4632.1989.tb00882.x}
  {\path{doi:https://doi.org/10.1111/j.1538-4632.1989.tb00882.x}}.

\bibitem{BGNR2023}
Micha{\l} Banaszak, Krzysztof G{\'o}rnisiewicz, Peter Nijkamp, and Waldemar
  Ratajczak.
\newblock Fractal dimension complexity of gravitation fractals in central place
  theory.
\newblock {\em Scientific Reports}, 13(1):2343, Feb 2023.
\newblock \href {https://doi.org/10.1038/s41598-023-28534-y}
  {\path{doi:10.1038/s41598-023-28534-y}}.

\bibitem{Bat2013}
Michael Batty.
\newblock {\em {The New Science of Cities}}.
\newblock The MIT Press, 11 2013.
\newblock \href {https://doi.org/10.7551/mitpress/9399.001.0001}
  {\path{doi:10.7551/mitpress/9399.001.0001}}.

\bibitem{BCZ2023}
Raghavendra~N. Bhat, Cristian Cobeli, and Alexandru Zaharescu.
\newblock On quasi-periodicity in {Proth-Gilbreath} triangles.
\newblock {\em arXiv preprint arXiv:2307.11776}, 2023.
\newblock \url{https://doi.org/10.48550/arXiv.2307.11776}.

\bibitem{Buy2022}
Seok~Hyun Byun.
\newblock Lozenge tilings of hexagons with holes on three crossing lines.
\newblock {\em Adv. Math.}, 398:22, 2022.
\newblock Id/No 108230.
\newblock \href {https://doi.org/10.1016/j.aim.2022.108230}
  {\path{doi:10.1016/j.aim.2022.108230}}.

\bibitem{CZZ2013}
Mihai Caragiu, Alexandru Zaharescu, and Mohammad Zaki.
\newblock An analogue of the {Proth}-{Gilbreath} conjecture.
\newblock {\em Far East J. Math. Sci. (FJMS)}, 81(1):1--12, 2013.
\newblock \url{http://www.pphmj.com/abstract/7973.htm}.

\bibitem{CP2023}
Cesar Ceballos and Viviane Pons.
\newblock The $s$-weak order and $s$-permutahedra {II:} the combinatorial
  complex of pure intervals.
\newblock {\em arXiv preprint arXiv:2309.14261v2}, 2023.
\newblock \url{https://doi.org/10.48550/arXiv.2309.14261}.

\bibitem{CECZ2001}
M.~Ciucu, T.~Eisenk{\"o}lbl, C.~Krattenthaler, and D.~Zare.
\newblock Enumeration of {Lozenge} tilings of hexagons with a central
  triangular hole.
\newblock {\em J. Comb. Theory, Ser. A}, 95(2):251--334, 2001.
\newblock \href {https://doi.org/10.1006/jcta.2000.3165}
  {\path{doi:10.1006/jcta.2000.3165}}.

\bibitem{Ciucu2005}
Mihai Ciucu.
\newblock {\em A random tiling model for two dimensional electrostatics},
  volume 839 of {\em Mem. Am. Math. Soc.}
\newblock Providence, RI: American Mathematical Society (AMS), 2005.
\newblock \href {https://doi.org/10.1090/memo/0839}
  {\path{doi:10.1090/memo/0839}}.

\bibitem{Ciucu2009}
Mihai Ciucu.
\newblock {\em The scaling limit of the correlation of holes on the triangular
  lattice with periodic boundary conditions}, volume 935 of {\em Mem. Am. Math.
  Soc.}
\newblock Providence, RI: American Mathematical Society (AMS), 2009.
\newblock \href {https://doi.org/10.1090/memo/0935}
  {\path{doi:10.1090/memo/0935}}.

\bibitem{CF2023}
Mihai Ciucu and Ilse Fischer.
\newblock Lozenge tilings of hexagons with removed core and satellites.
\newblock {\em Ann. Inst. Henri {Poincaré} Comb. Phys. Interact.},
  10(3):407--501, 2023.
\newblock \href {https://doi.org/10.4171/AIHPD/131}
  {\path{doi:10.4171/AIHPD/131}}.

\bibitem{CL2019}
Mihai Ciucu and Tri Lai.
\newblock Lozenge tilings of doubly-intruded hexagons.
\newblock {\em J. Comb. Theory, Ser. A}, 167:294--339, 2019.
\newblock \href {https://doi.org/10.1016/j.jcta.2019.05.004}
  {\path{doi:10.1016/j.jcta.2019.05.004}}.

\bibitem{CLR2021}
Mihai Ciucu, Tri Lai, and Ranjan Rohatgi.
\newblock Tilings of hexagons with a removed triad of bowties.
\newblock {\em J. Comb. Theory, Ser. A}, 178:40, 2021.
\newblock Id/No 105359.
\newblock \href {https://doi.org/10.1016/j.jcta.2020.105359}
  {\path{doi:10.1016/j.jcta.2020.105359}}.

\bibitem{CCZ2000}
C.~I. Cobeli, M.~Cr{\^a}{\c{s}}maru, and A.~Zaharescu.
\newblock A cellular automaton on a torus.
\newblock {\em Port. Math.}, 57(3):311--323, 2000.
\newblock \url{https://www.emis.de/journals/PM/57f3/pm57f305.pdf}.

\bibitem{CPZ2016}
Cristian Cobeli, Mihai Prunescu, and Alexandru Zaharescu.
\newblock A growth model based on the arithmetic {{\(Z\)}}-game.
\newblock {\em Chaos Solitons Fractals}, 91:136--147, 2016.
\newblock \href {https://doi.org/10.1016/j.chaos.2016.05.016}
  {\path{doi:10.1016/j.chaos.2016.05.016}}.

\bibitem{CZ2013}
Cristian Cobeli and Alexandru Zaharescu.
\newblock Promenade around {Pascal} triangle -- number motives.
\newblock {\em Bull. Math. Soc. Sci. Math. Roum., Nouv. S{\'e}r.},
  56(1):73--98, 2013.
\newblock \url{https://www.jstor.org/stable/43679285}.

\bibitem{CZ2014}
Cristian Cobeli and Alexandru Zaharescu.
\newblock A game with divisors and absolute differences of exponents.
\newblock {\em J. Difference Equ. Appl.}, 20(11):1489--1501, 2014.
\newblock \href {https://doi.org/10.1080/10236198.2014.940337}
  {\path{doi:10.1080/10236198.2014.940337}}.

\bibitem{CD2014}
Robert Connelly and William Dickinson.
\newblock Periodic planar disc packings.
\newblock {\em Philos. Trans. R. Soc. Lond., Ser. A, Math. Phys. Eng. Sci.},
  372(2008):17, 2014.
\newblock Id/No 20120039.
\newblock \href {https://doi.org/10.1098/rsta.2012.0039}
  {\path{doi:10.1098/rsta.2012.0039}}.

\bibitem{CRS1999}
J.~H. Conway, E.~M. Rains, and N.~J.~A. Sloane.
\newblock On the existence of similar sublattices.
\newblock {\em Can. J. Math.}, 51(6):1300--1306, 1999.
\newblock \href {https://doi.org/10.4153/CJM-1999-059-5}
  {\path{doi:10.4153/CJM-1999-059-5}}.

\bibitem{Gil2011}
Norman Gilbreath.
\newblock Processing process: the {Gilbreath} conjecture.
\newblock {\em J. Number Theory}, 131(12):2436--2441, 2011.
\newblock \url{https://doi.org/10.1016/j.jnt.2011.06.008}.
\newblock \href {https://doi.org/10.1016/j.jnt.2011.06.008}
  {\path{doi:10.1016/j.jnt.2011.06.008}}.

\bibitem{Gol1935}
M.~Goldberg.
\newblock A class of multi-symmetric polyhedra.
\newblock {\em Bull. Am. Math. Soc.}, 41:783, 1935.
\newblock
  {\url{https://www.ams.org/journals/bull/1935-41-11/S0002-9904-1935-06194-4/S0002-9904-1935-06194-4.pdf}}.

\bibitem{Gol1937}
M.~Goldberg.
\newblock A class of multi-symmetric polyhedra.
\newblock {\em T{\^o}hoku Math. J.}, 43:104--108, 1937.
\newblock
  {\url{https://www.jstage.jst.go.jp/article/tmj1911/43/0/43_0_104/_pdf}}.

\bibitem{Guy1988}
Richard~K. Guy.
\newblock The strong law of small numbers.
\newblock {\em Am. Math. Mon.}, 95(8):697--712, 1988.
\newblock \url{https://doi.org/10.2307/2322249}.
\newblock \href {https://doi.org/10.2307/2322249} {\path{doi:10.2307/2322249}}.

\bibitem{Guy2004}
Richard~K. Guy.
\newblock {\em Unsolved problems in number theory}.
\newblock Probl. Books Math. New York, NY: Springer-Verlag, 3rd ed. edition,
  2004.

\bibitem{KRNG2024}
Jerzy Kaczorowski, Waldemar Ratajczak, Peter Nijkamp, and Krzysztof
  Górnisiewicz.
\newblock Economic hierarchical spatial systems -- new properties of
  {Löschian} numbers.
\newblock {\em Applied Mathematics and Computation}, 461:128319, 2024.
\newblock \href {https://doi.org/https://doi.org/10.1016/j.amc.2023.128319}
  {\path{doi:https://doi.org/10.1016/j.amc.2023.128319}}.

\bibitem{Loc1940}
August Lösch.
\newblock {\em Economics of location}.
\newblock Yale University Press, 1954.
\newblock {\url{https://archive.org/details/economicsoflocat00ls/page/109}}.

\bibitem{Mar1977}
John~U. Marshall.
\newblock Christallerian networks in the {Löschian} economic landscape.
\newblock {\em The Professional Geographer}, 29(2):153--159, 1977.
\newblock \href {https://doi.org/10.1111/j.0033-0124.1977.00153.x}
  {\path{doi:10.1111/j.0033-0124.1977.00153.x}}.

\bibitem{oeis}
{OEIS Foundation Inc\!\!}
\newblock The {O}n-{L}ine {E}ncyclopedia of {I}nteger {S}equences, 2023.
\newblock Published electronically at \url{http://oeis.org}.

\bibitem{Pet1998}
Ivars Peterson.
\newblock {\em The jungles of randomness. {A} mathematical safari}.
\newblock New York, NY: Wiley, 1998.

\bibitem{Pro1878}
F.~Proth.
\newblock Sur la s{\'e}rie des nombres premiers.
\newblock {\em Nouvelle Correspondance Math{\'e}matique}, 4:236--240, 1878.
\newblock
  \url{https://gdz.sub.uni-goettingen.de/download/pdf/PPN598948236\_0004/LOG\_0088.pdf}.

\bibitem{Pru2022}
Mihai Prunescu.
\newblock Symmetries in the {Pascal} triangle: {{\(p\)}}-adic valuation,
  sign-reduction modulo {{\(p\)}} and the last non-zero digit.
\newblock {\em Bull. Math. Soc. Sci. Math. Roum., Nouv. S{\'e}r.},
  65(4):431--447, 2022.
\newblock \url{https://ssmr.ro/bulletin/pdf/65-4/articol_6.pdf}.

\bibitem{Rus2017}
Jacob Rus.
\newblock Flowsnake {Earth}.
\newblock In David Swart, Carlo~H. S\'{e}quin, and Krist\'{o}f Fenyvesi,
  editors, {\em Proceedings of {Bridges 2017: Mathematics, Art, Music,
  Architecture, Education, Culture}}, pages 237--244, Phoenix, Arizona, 2017.
  Tessellations Publishing.
\newblock URL:
  \url{http://archive.bridgesmathart.org/2017/bridges2017-237.html}.

\bibitem{RG2021}
Paul Tabatabai and Dieter~P. Gruber.
\newblock Knights and liars on graphs.
\newblock {\em J. Integer Seq.}, 24(5):article 21.5.8, 27, 2021.
\newblock URL: \url{cs.uwaterloo.ca/journals/JIS/VOL24/Tabatabai/taba4.html}.

\end{thebibliography}

\end{document}